\documentclass[12pt]{article}

\usepackage{latexsym,amsfonts,amsthm,amsmath,amscd,amssymb}
\usepackage{graphicx,enumerate,fancyhdr,caption2}

\newtheorem{lem}{Lemma}[section]
\newtheorem{thm}[lem]{Theorem}
\newtheorem{rem}[lem]{Remark}
\newtheorem{prop}[lem]{Proposition}

\newtheorem{conj}[lem]{Conjecture}

\newtheorem{defn}[lem]{Definition}

\newcommand{\color}[2][1]{}

\reversemarginpar

\newcommand{\dimo}[1]{\vspace{4pt}\noindent\textit{Proof of \ref{#1}}.\ }
\newcommand{\finedimo}{{\hfill\hbox{$\square$}\vspace{4pt}}}
\newcommand{\myprod}{\!\cdot\!}

\newcommand\matN{{\mathbb{N}}}

\newcommand\matS{{\mathbb{S}}}

\newcommand\matC{{\mathbb{C}}}
\newcommand\matE{{\mathbb{E}}}
\newcommand\matH{{\mathbb{H}}}

\newcommand{\myto}{\mathop{\longrightarrow}\limits}

\renewcommand{\hbar}{{\overline{h}}}

\newfont{\Got}{eufm10 scaled 1200}
\newcommand{\permu}{{\hbox{\Got S}}}

\newcommand{\mycap} [1] {\caption{\footnotesize{#1}}}

\newcommand{\chiorb}{\chi^{{\mathrm{orb}}}}

\newcommand{\Sigmatil}{\widetilde\Sigma}
\newcommand{\Xtil}{\widetilde X}

\newcommand{\dotsto}{\mathop{\dashrightarrow}\limits}

\newcommand{\argdotsto}[2]{\mathop{\dashrightarrow}\limits^{#1}_{\scriptscriptstyle{#2}}}
\newcommand{\argdotstobis}[2]{\mathop{\dashrightarrow\dashrightarrow}\limits^{#1}_{\scriptscriptstyle{#2}}}
\newcommand{\argdotstoter}[2]{\mathop{\dashrightarrow\dashrightarrow\dashrightarrow}\limits^{#1}_{\scriptscriptstyle{#2}}}
\newcommand{\argdotstoqua}[2]{\mathop{\dashrightarrow\dashrightarrow\dashrightarrow\dashrightarrow}\limits^{#1}_{\scriptscriptstyle{#2}}}

\newcommand{\argdotstosex}[2]{\mathop{\dashrightarrow\dashrightarrow\dashrightarrow\dashrightarrow\dashrightarrow
    \dashrightarrow}\limits^{#1}_{\scriptscriptstyle{#2}}}

\title{Branched covers of the sphere\\ and the prime-degree conjecture}

\author{Maria Antonietta~\textsc{Pascali}
\and\addtocounter{footnote}{5} Carlo~\textsc{Petronio}}

\begin{document}

\maketitle

\begin{abstract}
\noindent To a branched cover $\Sigmatil\to\Sigma$ between closed, connected and
orientable surfaces one associates a \emph{branch datum}, which
consists of $\Sigma$ and $\Sigmatil$, the total degree $d$, and
the partitions of $d$ given by the collections of local degrees over the branching points.
This datum must satisfy the Riemann-Hurwitz formula.
A \emph{candidate surface cover} is an abstract branch datum, a
priori not coming from a branched cover, but
satisfying the Riemann-Hurwitz formula.
The old \emph{Hurwitz problem} asks which candidate surface covers are realizable
by branched covers. It is now known that all candidate covers are realizable
when $\Sigma$ has positive genus, but not all are when $\Sigma$ is the $2$-sphere.
However a long-standing conjecture asserts that candidate covers with prime degree
are realizable. To a candidate surface cover one can associate one
$\Xtil\dotsto X$ between $2$-orbifolds, and in~\cite{PaPe} we have
completely analyzed the candidate surface covers such that either
$X$ is bad, spherical, or Euclidean, or both $X$ and $\Xtil$ are
rigid hyperbolic orbifolds, thus also providing strong supporting
evidence for the prime-degree conjecture. In this paper, using a variety
of different techniques, we continue this analysis,
carrying it out completely for the case where
$X$ is hyperbolic and rigid and $\Xtil$ has a $2$-dimensional Teichm\"uller space.
We find many more realizable and non-realizable candidate covers,
providing more support for the prime-degree conjecture.
\end{abstract}

\maketitle

\noindent
In this paper we push one step forward the approach via geometric $2$-orbifolds
developed and first exploited in~\cite{PaPe} to face the
Hurwitz existence problem for branched covers
between surfaces. See~\cite{Hurwitz} for the original
source concerning this problem, the
classical~\cite{Husemoller,Thom,Singerman,Endler,Francis,Ezell,Mednykh1,EKS,Gersten,Mednykh2,KZ},
the more
recent~\cite{Baranski,MSS,OkPand,PaPe,PePe1,PePe2,Pako,Zheng}, and below.
In~\cite{PaPe} we have determined the realizability of all candidate surface branched covers
having associated candidate cover between $2$-orbifolds with non-negative Euler
characteristic or between rigid hyperbolic $2$-orbifolds. These results provided
in particular strong support for the
long-standing conjecture~\cite{EKS} that a candidate cover with
prime total degree is always realizable.
In this paper we consider the case where in the associated
candidate cover between $2$-orbifolds the covered $2$-orbifold is hyperbolic and rigid
while the covering $2$-orbifold has a $2$-dimensional Teichm\"uller space.
As a result we exhibit many new realizable and non-realizable candidate surface branched covers,
finding a confirmation of the validity of the prime-degree conjecture for the
cases under consideration.
To discuss realizability we use a variety of techniques developed over the time by several
authors, and some new ones.
Some of the results proved in this paper are also contained, with minor variations, in the PhD thesis~\cite{tesi}
of the first named author (University of Rome I, 2010).

The paper is organized as follows. In Section~\ref{known:pre-orb:sec}
we precisely state the Hurwitz problem, fixing the notation we employ to treat it,
and we outline the classical results that motivate us to restrict our attention to
one specific instance of the problem.
In Section~\ref{known:orb:sec} we discuss how geometric $2$-orbifolds relate to the problem
and allow to split it into some more specific subproblems, stating the results of~\cite{PaPe}
where the easiest of these subproblems were solved.
In Section~\ref{new:sec} we describe the next easiest subproblem of the Hurwitz existence problem,
which is treated in this paper, and
we state the corresponding solution. This subproblem itself splits into
two different cases, for each of which one needs accomplish two tasks, namely to enumerate
the relevant candidate covers and to discuss their realizability.
The required enumeration processes for the two cases are carried
out in Section~\ref{enum:sec}.
Next, Section~\ref{techniques:sec} contains an overview of the different techniques
later used to discuss the realizability of the candidate covers found.
The discussion itself, again separately for the two relevant cases,
is finally carried out in Section~\ref{real:ex:sec}.

\section{(Candidate) surface branched covers,\\ and some known results}
\label{known:pre-orb:sec}

In this section we state the Hurwitz existence problem using the same
language and notation as in~\cite{PaPe}, and we very briefly review some by now classical results.

\paragraph{Branched covers}
Let $\Sigmatil$ and $\Sigma$ be closed, connected, and orientable surfaces, and
$f:\Sigmatil\to\Sigma$ be a branched cover, \emph{i.e.}, a map
locally modelled on functions of the form $(\matC,0)\myto^{z\mapsto z^k}(\matC,0)$
with $k\geqslant 1$. If $k>1$ then $0$ in the target $\matC$ is  a
branching point, and $k$ is the local degree at $0$ in the
source $\matC$. There is a finite number $n$ of branching points,
and, removing all of them from $\Sigma$ and their preimages
from $\Sigmatil$, we see that $f$ induces a genuine cover
of some degree $d$. The collection $(d_{ij})_{j=1}^{m_i}$ of
the local degrees at the preimages of the $i$-th branching point is a
partition $\Pi_i$ of $d$.
We now define:
\begin{itemize}
\item $\ell(\Pi_i)$ to be the length $m_i$ of $\Pi_i$;
\item $\Pi$ as the set $\{\Pi_1,\ldots,\Pi_n\}$ of all partitions of $d$ associated to $f$;
\item $\ell(\Pi)$ to be the total length $\ell(\Pi_1)+\ldots+\ell(\Pi_n)$ of $\Pi$.
\end{itemize}
Then multiplicativity of the Euler characteristic $\chi$ under
genuine covers for surfaces with boundary implies
the classical Riemann-Hurwitz formula
\begin{equation}\label{RH:general:eq}
\chi(\Sigmatil)-\ell(\Pi)=d\myprod\big(\chi(\Sigma)-n\big).
\end{equation}

\paragraph{Candidate branched covers and the realizability problem}
Consider again two closed, connected, and orientable surfaces $\Sigmatil$ and $\Sigma$,
integers $d\geqslant 2$ and $n\geqslant 1$, and a set of partitions
$\Pi=\{\Pi_1,\ldots,\Pi_n\}$ of $d$, with $\Pi_i=(d_{ij})_{j=1}^{m_i}$, such
that condition~(\ref{RH:general:eq}) is satisfied. We associate to these data
the symbol
$$\Sigmatil\argdotstosex{d:1}{(d_{11},\ldots,d_{1m_1}),\ldots,(d_{n1},\ldots,d_{nm_n})}\Sigma$$
that we will call a \emph{candidate surface branched cover}. A classical (and
still not completely solved) problem, known as the \emph{Hurwitz existence problem},
asks which candidate surface branched covers are actually \emph{realizable},
namely induced by some existent branched cover $f:\Sigmatil\to\Sigma$.
A non-realizable candidate surface branched cover will be called \emph{exceptional}.

Over the last 50 years the Hurwitz existence problem was the object of a wealth of papers,
many of which were listed above. The combined efforts of several
mathematicians led in particular to the following
results~\cite{Husemoller,EKS}:
\begin{itemize}
\item If $\chi(\Sigma)\leqslant 0$ then any candidate surface branched cover is realizable,
\emph{i.e.}, the Hurwitz existence problem has a positive solution in this case;
\item If $\chi(\Sigma)>0$, \emph{i.e.}, if $\Sigma$ is the $2$-sphere $S$, there exist
exceptional candidate surface branched covers.
\end{itemize}

\begin{rem}
\emph{A version of the Hurwitz existence problem exists also for
possibly non-orientable $\Sigmatil$ and $\Sigma$. Condition~(\ref{RH:general:eq})
must be complemented in this case with a few more requirements
(some obvious, and one slightly less obvious, see~\cite{PePe1}).
However it has been shown~\cite{Ezell,EKS}
that again this generalized problem always has a positive solution if $\chi(\Sigma)\leqslant 0$,
and that the case where $\Sigma$ is the projective plane reduces to the case
where $\Sigma$ is the $2$-sphere $S$.}
\end{rem}

According to the two facts stated, in order to face the Hurwitz existence problem,
it is not restrictive to \emph{assume the candidate covered surface $\Sigma$ is the $2$-sphere $S$},
which we will do henceforth.
Considerable energy has been devoted over the time to a general understanding of
the exceptional candidate surface branched covers in this case, and
quite some progress has been made (see for instance the survey of known results
contained in~\cite{PePe1}, together with the later papers~\cite{PePe2,Pako,Zheng}),
but the global pattern remains elusive. In particular the following conjecture
proposed in~\cite{EKS} appears to be still open:

\begin{conj}\label{prime:conj}
If $\Sigmatil\argdotsto{d:1}{\Pi}S$ is a candidate surface branched cover and
the degree $d$ is a prime number then the candidate is realizable.
\end{conj}

The following fact, again established in~\cite{EKS}, will serve to us as a motivation:

\begin{prop}\label{n=3:enough:prop}
If Conjecture~\emph{\ref{prime:conj}} is true for candidate surface branched covers having $n=3$ branching points,
then it is true in general.
\end{prop}

We conclude this section by mentioning that
all exceptional candidate surface branched covers
with $n=3$ and $d\leqslant 20$
were determined by computer in~\cite{Zheng}. There are very many of them, but
none occurs for prime $d$.

\section{Surface covers vs. $2$-orbifold covers,\\ and more known results}
\label{known:orb:sec}

A 2-orbifold $X=\Sigma(p_1,\ldots,p_n)$ is a closed orientable surface $\Sigma$
with $n$ cone points of orders $p_i\geqslant 2$, at which $X$ has
a singular differentiable structure given by the quotient
$\matC/_{\langle{\rm rot}(2\pi/p_i)\rangle}$, where ${\rm rot}(\vartheta):z\mapsto e^{i\vartheta}\cdot z$.

\paragraph{Geometric $2$-orbifolds}
W.~Thurston~\cite{thurston:notes} introduced the notion of orbifold Euler characteristic
$$\chiorb\big(\Sigma(p_1,\ldots,p_n)\big)=\chi(\Sigma)-\sum_{i=1}^n\left(1-\frac1{p_i}\right),$$
and showed that:
\begin{itemize}
\item If $\chiorb(X)>0$ then $X$ is either \emph{bad} (not covered by a surface in the sense of
orbifolds, see below) or
\emph{spherical}, namely the quotient of the metric 2-sphere $\matS^2$ under a finite isometric action;
\item If $\chiorb(X)=0$ (respectively, $\chiorb(X)<0$) then $X$ is
\emph{Euclidean} (respectively, \emph{hyperbolic}), namely
the quotient of the Euclidean plane $\matE^2$ (respectively, the hyperbolic plane $\matH^2$)
under a discrete isometric action.
\end{itemize}

In addition, Thurston proved that, for a hyperbolic $X$ with $n$ cone points and
underlying surface of genus $g$, the Teichm\"uller space
$\tau(X)$, namely the space of hyperbolic structures on $X$ up to
isometries isotopic to the identity, has real dimension
$6(g-1)+2n$.

\paragraph{Orbifold covers}
Again following Thurston~\cite{thurston:notes} we call
\emph{degree-$d$ orbifold cover} a map $f:\Xtil\to X$ between 2-orbifolds
such that $f^{-1}(x)$ generically consists of $d$ points and
locally making a diagram of the following form commutative:
$$\begin{array}{ccc}
(\matC,0) & \myto^{\rm id} & (\matC,0) \\
\downarrow & & \downarrow \\
(\Xtil,\widetilde{x}) & \myto^{f} & (X,x)
\end{array}$$
where $\widetilde{x}$ and $x$ have cone orders $\widetilde{p}$
and $p=k\myprod\widetilde{p}$ respectively, and the vertical arrows are the
projections corresponding to the actions of
$\langle{\rm rot}(2\pi/\widetilde{p})\rangle$ and $\langle{\rm rot}(2\pi/p)\rangle$,
namely the maps defining the (possibly singular) local differentiable structures at
$\widetilde{x}$ and $x$. Since this local model can be described by the map $z\mapsto z^k$,
we see that $f$ induces a branched cover between the underlying surfaces of $\Xtil$ and $X$.
Using the orbifold language one can then state the Riemann-Hurwitz formula~(\ref{RH:general:eq})
in the following equivalent fashion:
\begin{equation}\label{RH:orb:eq}
\chiorb(\Xtil)=d\myprod\chiorb(X).
\end{equation}

\paragraph{From surface to $2$-orbifold candidate covers}
As one easily sees, distinct orbifold covers can induce the same surface
branched cover (in the local model, the two cone orders can be multiplied by one and the
same integer). However, as pointed out in~\cite{PePe1},
a surface branched cover has an ``easiest'' associated
orbifold cover, \emph{i.e.}, that with the smallest possible cone orders. This carries over
to \emph{candidate} covers, as we will now spell out.
Consider a candidate surface branched cover
$$\Sigmatil\argdotstosex{d:1}{(d_{11},\ldots,d_{1m_1}),\ldots,(d_{n1},\ldots,d_{nm_n})}\Sigma$$
and define
$$\begin{array}{ll}
p_i={\rm l.c.m.}\{d_{ij}:\ j=1,\ldots,m_i\},\quad & p_{ij}=p_i/d_{ij},\phantom{\Big|} \\
X=\Sigma(p_1,\ldots,p_n), & \Xtil=\Sigmatil\big((p_{ij})_{i=1,\ldots,n}^{j=1,\ldots,m_i}\big)
\end{array}$$
where ``l.c.m.'' stands for ``least common multiple.''
Then we have a \emph{preferred associated candidate $2$-orbifold cover} $\Xtil\dotsto^{d:1}X$
satisfying $\chiorb(\Xtil)=d\cdot \chiorb(X)$. Note that the original
candidate surface branched cover cannot be reconstructed from $\Xtil,X,d$ alone,
but it can if $\Xtil\dotsto^{d:1}X$ is complemented with the \emph{covering instructions}
$$(p_{11},\ldots,p_{1m_1})\dotsto p_1,\qquad\ldots\qquad
(p_{n1},\ldots,p_{nm_n})\dotsto p_n$$
that one can include in the symbol $\Xtil\dotsto^{d:1}X$ itself,
omitting the $p_{ij}$'s equal to $1$. Of course
a candidate surface branched cover is realizable if and only if the
associated candidate 2-orbifold cover with appropriate covering instructions is realizable.

\paragraph{Splitting the Hurwitz problem according to orbifold geometry}
We have just shown that to each candidate surface branched cover one can
attach a preferred candidate orbifold cover $\Xtil\dotsto X$. Following Thurston's
geometric picture of $2$-orbifolds one can then split
the Hurwitz existence problem by restricting to the analysis of those
candidate surface branched covers for which in the associated
$\Xtil\dotsto X$ the geometry of $X$ and $\Xtil$ is prescribed, and
in the hyperbolic case the dimension of $\tau(X)$ and $\tau(\Xtil)$ also is.
Note that a candidate orbifold cover $\Xtil\dotsto X$ by definition
satisfies the orbifold version~(\ref{RH:orb:eq}) of the
Riemann-Hurwitz formula; therefore, $X$ and $\Xtil$ have the same geometry,
except possibly when one of them is bad and the other one is spherical.

Having in mind to attack Conjecture~\ref{prime:conj} and taking into account
Proposition~\ref{n=3:enough:prop}, one can actually restrict to the case where
\begin{itemize}
\item $X$ is a triangular orbifold $S(p,q,r)$, whence geometrically rigid,
\end{itemize}
and split the Hurwitz existence problem as described in the following table.
\begin{center}
\begin{tabular}{c|c}
\textbf{Subproblem} & \textbf{Description}\\ \hline\hline
bad/$\matS$ & $\phantom{\Big|}X$ is bad or spherical, namely $\chiorb(X)>0$\\ \hline
$\matE$ & $\phantom{\Big|}X$ is Euclidean, namely $\chiorb(X)=0$\\ \hline
$\matH(j)$ for $j\in\matN$ & $\phantom{\Big|}X$ is hyperbolic and $\dim(\tau(\Xtil))=2j$
\end{tabular}
\end{center}

\paragraph{Known result}
In~\cite{PaPe} we have completely solved the subproblems of the Hurwitz existence problem
described above as bad/$\matS$, as $\matE$, and as $\matH(0)$,
finding the results described in the following table.

\begin{center}
\begin{tabular}{c|c}
\textbf{Subproblem} & \textbf{Findings}\\ \hline\hline
bad/$\matS$ &
    \begin{minipage}{10cm}
    \ \\
    $20$ realizable isolated candidate covers, two infinite families
    of realizable candidate covers, 11 exceptional isolated candidate covers, and one infinite family
    of exceptional candidate covers\\
    \
    \end{minipage} \\ \hline
$\matE$ &
    \begin{minipage}{10cm}
    \ \\
    $14$ infinite families of realizable candidate covers, two infinite families of
    exceptional candidate covers, and $12$ infinite families of candidate covers for which the realizability
    was shown to be equivalent to an arithmetic condition on the degree\\
    \
    \end{minipage} \\ \hline
$\matH(0)$ &
    \begin{minipage}{10cm}
    \ \\
    $9$ realizable isolated candidate covers and two isolated exceptional candidate covers\\
    \
    \end{minipage}
\end{tabular}
\end{center}

We mention that the arithmetic conditions on the degree $d$ in the solution of subproblem $\matE$
are given by congruences and/or by the fact that $d$ belongs to the image of some
quadratic form $\matN\times\matN\to\matN$.
Analyzing these conditions,
for three of the families in the solution of subproblem $\matE$
we were able to show that the candidate cover is ``exceptional with probability 1,''
even though it is realizable when its degree is prime, which we view as a strong supporting evidence for
Conjecture~\ref{prime:conj}.

\section{New results}
\label{new:sec}

We describe here the new contribution offered by the present paper, namely
the solution of subproblem $\matH(1)$ of the Hurwitz existence problem.
This consists of the following steps:
\begin{itemize}
\item Enumeration of all the candidate surface branched covers having an associated
candidate orbifold
cover of the form $\Xtil\dotsto X$ where $X=S(p,q,r)$ is a rigid hyperbolic
$2$-orbifold, namely such that $\frac1p+\frac1q+\frac1r<1$,
and $\Xtil$ is hyperbolic with $\textrm{dim}(\tau(\Xtil))=2$,
namely $\Xtil=S(\alpha,\beta,\gamma,\delta)$ with $(\alpha,\beta,\gamma,\delta)\neq (2,2,2,2)$,
or $\Xtil=T(\alpha)$ with $\alpha>1$, where $T$ is the torus;
\item Discussion of realizability or exceptionality of all the candidate surface covers enumerated.
\end{itemize}
The next statements summarize our results.
Before giving them we underline that
all the exceptional candidates we have found occur for composite
degree, which means that our results provide further supporting
evidence for Conjecture~\ref{prime:conj}.
We also mention that the tables of the next pages contain a numbering of all the
candidate covers we have found; these numbers will be referred to throughout the paper.

\begin{thm}\label{hyp:S(4)-to-S(3):summary:thm}
There exist precisely $146$ candidate surface branched covers
$\Sigmatil\argdotstoter{d:1}{(\Pi_1,\Pi_2,\Pi_3)}S$ for which
in the associated candidate orbifold cover
$\Xtil\dotsto X$ one has that $X$ and $\Xtil$ are hyperbolic of the form
$X=S(p,q,r)$ and $\Xtil=S(\alpha,\beta,\gamma,\delta)$.
Precisely $29$ of these candidate surface covers are exceptional, and the other
$117$ are realizable. The complete description of these candidate covers, including
the associated candidate orbifold covers and information on their realizability,
is contained in Tables~\ref{hyp:S(4)-to-S(3):summary:1-12:tab} to~\ref{hyp:S(4)-to-S(3):summary:117-146:tab}.
\end{thm}

\begin{thm}\label{hyp:T(1)-to-S(3):summary:thm}
There exist precisely $22$ candidate surface branched covers
$\Sigmatil\argdotstoter{d:1}{(\Pi_1,\Pi_2,\Pi_3)}S$ for which
in the associated candidate orbifold cover
$\Xtil\dotsto X$ one has that $X$ and $\Xtil$ are hyperbolic of the form
$X=S(p,q,r)$ and $\Xtil=T(\alpha)$.
Precisely $5$ of these candidate surface covers are exceptional, and the other
$17$ are realizable. The complete description of these candidate covers, including
the associated candidate orbifold covers and information on their realizability,
is contained in Table~\ref{hyp:T(1)-to-S(3):summary:tab}.
\end{thm}

\begin{rem}
\emph{The issue of enumerating the relevant candidate covers for
Theorems~\ref{hyp:S(4)-to-S(3):summary:thm} and~\ref{hyp:T(1)-to-S(3):summary:thm}
is an elementary, though complicated, combinatorial problem, and its solution
presented below in Section~\ref{enum:sec} does not employ sophisticated techniques.
On the other hand, to discuss realizability of the candidates found,
we describe in general terms in Section~\ref{techniques:sec}
and then we exploit in Section~\ref{real:ex:sec} a variety of different
geometric methods. As a matter of fact, in some cases we offer new proofs
of the known realizability or exceptionality of some candidates,
and we establish the previously unknown realizability or exceptionality for other candidates
using two or more different methods. Our aim here is to show that a wealth of
different techniques are already in place and should allow one to attack
more and more advanced instances of the Hurwitz existence problem,
according to its splitting we have proposed above in Section~\ref{known:orb:sec}.}
\end{rem}

\begin{table}[h]
\begin{center}
\begin{tabular}{c|c|c|c|c|c|c}
$d$&$\Pi_1$&$\Pi_2$&$\Pi_3$&\footnotesize{Associated} $\Xtil\dotsto X$&\!\!\footnotesize{Realizable?}\!\!&$\#$\\ \hline\hline
\footnotesize{5}
&\footnotesize{(2,1,1,1)}&\footnotesize{(4,1)}&\footnotesize{(5)}&\!\footnotesize{$S(2,2,2,4)\dotsto S(2,4,5)$}\!&
    $\checkmark$&\footnotesize{\texttt{1}}\\ \cline{2-7}
&\footnotesize{(3,1,1)}&\footnotesize{(3,1,1)}&\footnotesize{(5)}&\!\footnotesize{$S(3,3,3,3)\dotsto S(3,3,5)$}\!&
    $\checkmark$&\footnotesize{\texttt{2}}\\ \cline{2-7}
&\footnotesize{(3,1,1)}&\footnotesize{(4,1)}&\footnotesize{(4,1)}&\!\footnotesize{$S(3,3,4,4)\dotsto S(3,4,4)$}\!&
    $\checkmark$&\footnotesize{\texttt{3}}\\ \cline{2-7}
&\footnotesize{(2,2,1)}&\footnotesize{(3,2)}&\footnotesize{(4,1)}&\!\footnotesize{$S(2,2,3,4)\dotsto S(2,4,6)$}\!&
    $\checkmark$&\footnotesize{\texttt{4}}\\ \hline

\footnotesize{6}
&\footnotesize{(2,2,1,1)}&\footnotesize{(4,1,1)}&\footnotesize{(6)}&\!\footnotesize{$S(2,2,4,4)\dotsto S(2,4,6)$}\!&
    $\checkmark$&\footnotesize{\texttt{5}}\\ \cline{2-7}
&\footnotesize{(2,2,1,1)}&\footnotesize{(5,1)}&\footnotesize{(5,1)}&\!\footnotesize{$S(2,2,5,5)\dotsto S(2,5,5)$}\!&
    $\checkmark$&\footnotesize{\texttt{6}}\\ \cline{2-7}
&\footnotesize{(2,2,1,1)}&\footnotesize{(4,2)}&\footnotesize{(5,1)}&\!\footnotesize{$S(2,2,2,5)\dotsto S(2,4,5)$}\!&
    $\checkmark$&\footnotesize{\texttt{7}}\\ \cline{2-7}
&\footnotesize{(3,1,1,1)}&\footnotesize{(3,3)}&\footnotesize{(5,1)}&\!\footnotesize{$S(3,3,3,5)\dotsto S(3,3,5)$}\!&
    $\checkmark$&\footnotesize{\texttt{8}}\\ \cline{2-7}
&\footnotesize{(3,1,1,1)}&\footnotesize{(3,3)}&\footnotesize{(4,2)}&\!\footnotesize{$S(2,3,3,3)\dotsto S(3,3,4)$}\!&
    $\checkmark$&\footnotesize{\texttt{9}}\\ \cline{2-7}
&\footnotesize{(3,3)}&\footnotesize{(4,1,1)}&\footnotesize{(4,1,1)}&\!\footnotesize{$S(4,4,4,4)\dotsto S(3,4,4)$}\!&
    $\checkmark$&\footnotesize{\texttt{10}}\\ \cline{2-7}
&\footnotesize{(2,2,2)}&\footnotesize{(3,2,1)}&\footnotesize{(5,1)}&\!\footnotesize{$S(2,3,5,6)\dotsto S(2,5,6)$}\!&
    $\checkmark$&\footnotesize{\texttt{11}}\\ \cline{2-7}
&\footnotesize{(2,2,2)}&\footnotesize{(3,2,1)}&\footnotesize{(4,2)}&\!\footnotesize{$S(2,2,3,6)\dotsto S(2,4,6)$}\!&
    $\checkmark$&\footnotesize{\texttt{12}}
\end{tabular}
\end{center}
\mycap{Candidate surface branched covers
$S\argdotstobis{d:1}{(\Pi_1,\Pi_2,\Pi_3)}S$
with associated hyperbolic $S(\alpha,\beta,\gamma,\delta)\dotsto S(p,q,r)$;
continued in Tables~\ref{hyp:S(4)-to-S(3):summary:13-47:tab}
to~\ref{hyp:S(4)-to-S(3):summary:117-146:tab}.\label{hyp:S(4)-to-S(3):summary:1-12:tab}}
\end{table}

\clearpage

\begin{table}
\begin{center}
\begin{tabular}{c|c|c|c|c|c|c}
$d$&$\Pi_1$&$\Pi_2$&$\Pi_3$&\footnotesize{Associated} $\Xtil\dotsto X$&\!\!\footnotesize{Realizable?}\!\!&$\#$\\ \hline\hline
\footnotesize{7}
&\footnotesize{(2,2,1,1,1)}&\footnotesize{(3,3,1)}&\footnotesize{(7)}&\!\footnotesize{$S(2,2,2,3)\dotsto S(2,3,7)$}\!&
    $\checkmark$&\footnotesize{\texttt{13}}\\ \cline{2-7}
&\footnotesize{(2,2,2,1)}&\footnotesize{(4,1,1,1)}&\footnotesize{(7)}&\!\footnotesize{$S(2,4,4,4)\dotsto S(2,4,7)$}\!&
    $\checkmark$&\footnotesize{\texttt{14}}\\ \cline{2-7}
&\footnotesize{(3,3,1)}&\footnotesize{(3,3,1)}&\footnotesize{(5,1,1)}&\!\footnotesize{$S(3,3,5,5)\dotsto S(3,3,5)$}\!&
    $\checkmark$&\footnotesize{\texttt{15}}\\ \cline{2-7}
&\footnotesize{(3,3,1)}&\footnotesize{(3,3,1)}&\footnotesize{(4,2,1)}&\!\footnotesize{$S(2,3,3,4)\dotsto S(3,3,4)$}\!&
    $\checkmark$&\footnotesize{\texttt{16}}\\ \cline{2-7}
&\footnotesize{(2,2,2,1)}&\footnotesize{(3,3,1)}&\footnotesize{(5,2)}&\!\footnotesize{$S(2,2,3,5)\dotsto S(2,3,10)$}\!&
    $\checkmark$&\footnotesize{\texttt{17}}\\ \cline{2-7}
&\footnotesize{(2,2,2,1)}&\footnotesize{(3,3,1)}&\footnotesize{(4,3)}&\!\footnotesize{$S(2,3,3,4)\dotsto S(2,3,12)$}\!&
    $\checkmark$&\footnotesize{\texttt{18}}\\ \cline{2-7}
&\footnotesize{(2,2,2,1)}&\footnotesize{(4,2,1)}&\footnotesize{(6,1)}&\!\footnotesize{$S(2,2,4,6)\dotsto S(2,4,6)$}\!&
    $\checkmark$&\footnotesize{\texttt{19}}\\ \cline{2-7}
&\footnotesize{(2,2,2,1)}&\footnotesize{(5,1,1)}&\footnotesize{(6,1)}&\!\footnotesize{$S(2,5,5,6)\dotsto S(2,5,6)$}\!&
    $\checkmark$&\footnotesize{\texttt{20}}\\ \hline

\footnotesize{8}
&\footnotesize{(2,2,2,2)}&\footnotesize{(4,1,1,1,1)}&\footnotesize{(8)}&\!\footnotesize{$S(4,4,4,4)\dotsto S(2,4,8)$}\!&
    $\checkmark$&\footnotesize{\texttt{21}}\\ \cline{2-7}
&\footnotesize{(2,2,2,1,1)}&\footnotesize{(3,3,1,1)}&\footnotesize{(8)}&\!\footnotesize{$S(2,2,3,3)\dotsto S(2,3,8)$}\!&
    $\checkmark$&\footnotesize{\texttt{22}}\\ \cline{2-7}
&\footnotesize{(2,2,2,2)}&\footnotesize{(3,2,2,1)}&\footnotesize{(4,4)}&\!\footnotesize{$S(2,3,3,6)\dotsto S(2,4,6)$}\!&
    \footnotesize{\textbf{Excep}}&\footnotesize{\texttt{23}}\\ \cline{2-7}
&\footnotesize{(2,2,2,2)}&\footnotesize{(5,1,1,1)}&\footnotesize{(7,1)}&\!\footnotesize{$S(5,5,5,7)\dotsto S(2,5,7)$}\!&
    $\checkmark$&\footnotesize{\texttt{24}}\\ \cline{2-7}
&\footnotesize{(2,2,2,2)}&\footnotesize{(5,1,1,1)}&\footnotesize{(6,2)}&\!\footnotesize{$S(3,5,5,5)\dotsto S(2,5,6)$}\!&
    \footnotesize{\textbf{Excep}}&\footnotesize{\texttt{25}}\\ \cline{2-7}
&\footnotesize{(2,2,2,2)}&\footnotesize{(4,2,1,1)}&\footnotesize{(7,1)}&\!\footnotesize{$S(2,4,4,7)\dotsto S(2,4,7)$}\!&
    $\checkmark$&\footnotesize{\texttt{26}}\\ \cline{2-7}
&\footnotesize{(2,2,2,2)}&\footnotesize{(4,2,1,1)}&\footnotesize{(6,2)}&\!\footnotesize{$S(2,3,4,4)\dotsto S(2,4,6)$}\!&
    $\checkmark$&\footnotesize{\texttt{27}}\\ \cline{2-7}
&\footnotesize{(2,2,2,2)}&\footnotesize{(3,3,1,1)}&\footnotesize{(5,3)}&\!\footnotesize{$S(3,3,3,5)\dotsto S(2,3,15)$}\!&
    \footnotesize{\textbf{Excep}}&\footnotesize{\texttt{28}}\\ \cline{2-7}
&\footnotesize{(3,3,1,1)}&\footnotesize{(3,3,1,1)}&\footnotesize{(4,4)}&\!\footnotesize{$S(3,3,3,3)\dotsto S(3,3,4)$}\!&
    $\checkmark$&\footnotesize{\texttt{29}}\\ \cline{2-7}
&\footnotesize{(2,2,2,2)}&\footnotesize{(6,1,1)}&\footnotesize{(6,1,1)}&\!\footnotesize{$S(6,6,6,6)\dotsto S(2,6,6)$}\!&
    $\checkmark$&\footnotesize{\texttt{30}}\\ \cline{2-7}
&\footnotesize{(2,2,2,2)}&\footnotesize{(4,2,2)}&\footnotesize{(6,1,1)}&\!\footnotesize{$S(2,2,6,6)\dotsto S(2,4,6)$}\!&
    \footnotesize{\textbf{Excep}}&\footnotesize{\texttt{31}}\\ \cline{2-7}
&\footnotesize{(2,2,2,1,1)}&\footnotesize{(4,4)}&\footnotesize{(6,1,1)}&\!\footnotesize{$S(2,2,6,6)\dotsto S(2,4,6)$}\!&
    $\checkmark$&\footnotesize{\texttt{32}}\\ \hline

\footnotesize{9}
&\!\!\footnotesize{(2,2,2,1,1,1)}\!\!&\footnotesize{(3,3,3)}&\footnotesize{(8,1)}&\!\footnotesize{$S(2,2,2,8)\dotsto S(2,3,8)$}\!&
    $\checkmark$&\footnotesize{\texttt{33}}\\ \cline{2-7}
&\footnotesize{(2,2,2,2,1)}&\footnotesize{(3,3,1,1,1)}&\footnotesize{(9)}&\!\footnotesize{$S(2,3,3,3)\dotsto S(2,3,9)$}\!&
    $\checkmark$&\footnotesize{\texttt{34}}\\ \cline{2-7}
&\footnotesize{(3,3,3)}&\footnotesize{(3,3,3)}&\!\!\footnotesize{(5,1,1,1,1)}\!\!&\!\footnotesize{$S(5,5,5,5)\dotsto S(3,3,5)$}\!&
    \footnotesize{\textbf{Excep}}&\footnotesize{\texttt{35}}\\ \cline{2-7}
&\footnotesize{(3,3,3)}&\footnotesize{(3,3,3)}&\!\!\footnotesize{(4,2,1,1,1)}\!\!&\!\footnotesize{$S(2,4,4,4)\dotsto S(3,3,4)$}\!&
    \footnotesize{\textbf{Excep}}&\footnotesize{\texttt{36}}\\ \cline{2-7}
&\footnotesize{(3,3,1,1,1)}&\footnotesize{(3,3,3)}&\footnotesize{(4,4,1)}&\!\footnotesize{$S(3,3,3,4)\dotsto S(3,3,4)$}\!&
    $\checkmark$&\footnotesize{\texttt{37}}\\ \cline{2-7}
&\footnotesize{(2,2,2,2,1)}&\footnotesize{(4,4,1)}&\footnotesize{(7,1,1)}&\!\footnotesize{$S(2,4,7,7)\dotsto S(2,4,7)$}\!&
    $\checkmark$&\footnotesize{\texttt{38}}\\ \cline{2-7}
&\footnotesize{(2,2,2,2,1)}&\footnotesize{(4,4,1)}&\footnotesize{(6,2,1)}&\!\footnotesize{$S(2,3,4,6)\dotsto S(2,4,6)$}\!&
    $\checkmark$&\footnotesize{\texttt{39}}\\ \cline{2-7}
&\footnotesize{(2,2,2,2,1)}&\footnotesize{(3,3,3)}&\footnotesize{(5,3,1)}&\!\footnotesize{$S(2,3,5,15)\dotsto S(2,3,15)$}\!&
    $\checkmark$&\footnotesize{\texttt{40}}\\ \cline{2-7}
&\footnotesize{(2,2,2,2,1)}&\footnotesize{(3,3,3)}&\footnotesize{(5,2,2)}&\!\footnotesize{$S(2,2,5,5)\dotsto S(2,3,10)$}\!&
    \footnotesize{\textbf{Excep}}&\footnotesize{\texttt{41}}\\ \cline{2-7}
&\footnotesize{(2,2,2,2,1)}&\footnotesize{(3,3,3)}&\footnotesize{(4,3,2)}&\!\footnotesize{$S(2,3,4,6)\dotsto S(2,3,12)$}\!&
    $\checkmark$&\footnotesize{\texttt{42}}\\ \hline
\footnotesize{10}
&\footnotesize{(2,2,2,2,2)}&\!\!\footnotesize{(3,3,1,1,1,1)}\!\!&\footnotesize{(10)}&\!\footnotesize{$S(3,3,3,3)\dotsto S(2,3,10)$}\!&
    $\checkmark$&\footnotesize{\texttt{43}}\\ \cline{2-7}
&\!\!\footnotesize{(2,2,2,2,1,1)}\!\!&\footnotesize{(4,4,1,1)}&\footnotesize{(5,5)}&\!\footnotesize{$S(2,2,4,4)\dotsto S(2,4,5)$}\!&
    $\checkmark$&\footnotesize{\texttt{44}}\\ \cline{2-7}
&\!\!\footnotesize{(2,2,2,2,1,1)}\!\!&\footnotesize{(3,3,3,1)}&\footnotesize{(8,2)}&\!\footnotesize{$S(2,2,3,4)\dotsto S(2,3,8)$}\!&
    $\checkmark$&\footnotesize{\texttt{45}}\\ \cline{2-7}
&\!\!\footnotesize{(2,2,2,2,1,1)}\!\!&\footnotesize{(3,3,3,1)}&\footnotesize{(9,1)}&\!\footnotesize{$S(2,2,3,9)\dotsto S(2,3,9)$}\!&
    $\checkmark$&\footnotesize{\texttt{46}}\\ \cline{2-7}
&\footnotesize{(2,2,2,2,2)}&\footnotesize{(5,5)}&\!\!\footnotesize{(6,1,1,1,1)}\!\!&\!\footnotesize{$S(6,6,6,6)\dotsto S(2,5,6)$}\!&
    $\checkmark$&\footnotesize{\texttt{47}}
\end{tabular}
\end{center}
\mycap{Continued from Table~\ref{hyp:S(4)-to-S(3):summary:1-12:tab}.\label{hyp:S(4)-to-S(3):summary:13-47:tab}}
\end{table}

\begin{table}
\begin{center}
\begin{tabular}{c|c|c|c|c|c|c}
$d$&$\Pi_1$&$\Pi_2$&$\Pi_3$&\footnotesize{Associated} $\Xtil\dotsto X$&\!\!\footnotesize{Realizable?}\!\!&$\#$\\ \hline\hline
\footnotesize{10}
&\footnotesize{(2,2,2,2,2)}&\footnotesize{(4,2,2,1,1)}&\footnotesize{(5,5)}&\!\footnotesize{$S(2,2,4,4)\dotsto S(2,4,5)$}\!&
    $\checkmark$&\footnotesize{\texttt{48}}\\ \cline{2-7}
&\footnotesize{(2,2,2,2,2)}&\footnotesize{(4,4,2)}&\footnotesize{(7,1,1,1)}&\!\footnotesize{$S(2,7,7,7)\dotsto S(2,4,7)$}\!&
    \footnotesize{\textbf{Excep}}&\footnotesize{\texttt{49}}\\ \cline{2-7}
&\footnotesize{(2,2,2,2,2)}&\footnotesize{(4,4,2)}&\footnotesize{(6,2,1,1)}&\!\footnotesize{$S(2,3,6,6)\dotsto S(2,4,6)$}\!&
    \footnotesize{\textbf{Excep}}&\footnotesize{\texttt{50}}\\ \cline{2-7}
&\footnotesize{(2,2,2,2,2)}&\footnotesize{(4,4,1,1)}&\footnotesize{(8,1,1)}&\!\footnotesize{$S(4,4,8,8)\dotsto S(2,4,8)$}\!&
    $\checkmark$&\footnotesize{\texttt{51}}\\ \cline{2-7}
&\footnotesize{(2,2,2,2,2)}&\footnotesize{(4,4,1,1)}&\footnotesize{(6,3,1)}&\!\footnotesize{$S(2,4,4,6)\dotsto S(2,4,6)$}\!&
    $\checkmark$&\footnotesize{\texttt{52}}\\ \cline{2-7}
&\footnotesize{(2,2,2,2,2)}&\footnotesize{(4,4,1,1)}&\footnotesize{(6,2,2)}&\!\footnotesize{$S(3,3,4,4)\dotsto S(2,4,6)$}\!&
    $\checkmark$&\footnotesize{\texttt{53}}\\ \cline{2-7}
&\footnotesize{(2,2,2,2,2)}&\footnotesize{(3,3,3,1)}&\footnotesize{(7,2,1)}&\!\footnotesize{$S(2,3,7,14)\dotsto S(2,3,14)$}\!&
    $\checkmark$&\footnotesize{\texttt{54}}\\ \cline{2-7}
&\footnotesize{(2,2,2,2,2)}&\footnotesize{(3,3,3,1)}&\footnotesize{(5,4,1)}&\!\footnotesize{$S(3,4,5,20)\dotsto S(2,3,20)$}\!&
    $\checkmark$&\footnotesize{\texttt{55}}\\ \cline{2-7}
&\footnotesize{(2,2,2,2,2)}&\footnotesize{(3,3,3,1)}&\footnotesize{(5,3,2)}&\!\footnotesize{$S(3,6,10,15)\dotsto S(2,3,30)$}\!&
    $\checkmark$&\footnotesize{\texttt{56}}\\ \cline{2-7}
&\footnotesize{(2,2,2,2,2)}&\footnotesize{(3,3,3,1)}&\footnotesize{(4,3,3)}&\!\footnotesize{$S(3,3,4,4)\dotsto S(2,3,12)$}\!&
    \footnotesize{\textbf{Excep}}&\footnotesize{\texttt{57}}\\ \cline{2-7}
&\footnotesize{(3,3,3,1)}&\footnotesize{(3,3,3,1)}&\footnotesize{(4,4,1,1)}&\!\footnotesize{$S(3,3,4,4)\dotsto S(3,3,4)$}\!&
    $\checkmark$&\footnotesize{\texttt{58}}\\ \hline

\footnotesize{11}
&\footnotesize{(2,2,2,2,2,1)}&\footnotesize{(3,3,3,1,1)}&\footnotesize{(10,1)}&\!\footnotesize{$S(2,3,3,10)\dotsto S(2,3,10)$}\!&
    $\checkmark$&\footnotesize{\texttt{59}}\\ \cline{2-7}
&\footnotesize{(2,2,2,2,2,1)}&\footnotesize{(4,4,2,1)}&\footnotesize{(5,5,1)}&\!\footnotesize{$S(2,2,4,5)\dotsto S(2,4,5)$}\!&
    $\checkmark$&\footnotesize{\texttt{60}}\\ \hline

\footnotesize{12}
&\!\!\footnotesize{(2,2,2,2,2,1,1)}\!\!&\footnotesize{(3,3,3,3)}&\footnotesize{(10,1,1)}&\!\footnotesize{$S(2,2,10,10)\dotsto S(2,3,10)$}\!&
    $\checkmark$&\footnotesize{\texttt{61}}\\ \cline{2-7}
&\!\!\footnotesize{(2,2,2,2,2,1,1)}\!\!&\footnotesize{(3,3,3,3)}&\footnotesize{(8,2,2)}&\!\footnotesize{$S(2,2,4,4)\dotsto S(2,3,8)$}\!&
    $\checkmark$&\footnotesize{\texttt{62}}\\ \cline{2-7}
&\!\!\footnotesize{(2,2,2,2,2,1,1)}\!\!&\footnotesize{(4,4,4)}&\footnotesize{(5,5,1,1)}&\!\footnotesize{$S(2,2,5,5)\dotsto S(2,4,5)$}\!&
    $\checkmark$&\footnotesize{\texttt{63}}\\ \cline{2-7}
&\footnotesize{(2,\ldots,2)}&\!\!\footnotesize{(4,4,1,1,1,1)}\!\!&\footnotesize{(6,6)}&\!\footnotesize{$S(4,4,4,4)\dotsto S(2,4,6)$}\!&
    $\checkmark$&\footnotesize{\texttt{64}}\\ \cline{2-7}
&\footnotesize{(2,\ldots,2)}&\!\!\footnotesize{(3,3,3,1,1,1)}\!\!&\footnotesize{(11,1)}&\!\footnotesize{$S(3,3,3,11)\dotsto S(2,3,11)$}\!&
    $\checkmark$&\footnotesize{\texttt{65}}\\ \cline{2-7}
&\footnotesize{(2,\ldots,2)}&\!\!\footnotesize{(3,3,3,1,1,1)}\!\!&\footnotesize{(10,2)}&\!\footnotesize{$S(3,3,3,5)\dotsto S(2,3,10)$}\!&
    $\checkmark$&\footnotesize{\texttt{66}}\\ \cline{2-7}
&\footnotesize{(2,\ldots,2)}&\!\!\footnotesize{(3,3,3,1,1,1)}\!\!&\footnotesize{(9,3)}&\!\footnotesize{$S(3,3,3,3)\dotsto S(2,3,9)$}\!&
    $\checkmark$&\footnotesize{\texttt{67}}\\ \cline{2-7}
&\footnotesize{(2,\ldots,2)}&\!\!\footnotesize{(3,3,3,1,1,1)}\!\!&\footnotesize{(8,4)}&\!\footnotesize{$S(2,3,3,3)\dotsto S(2,3,8)$}\!&
    $\checkmark$&\footnotesize{\texttt{68}}\\ \cline{2-7}
&\footnotesize{(2,\ldots,2)}&\footnotesize{(4,4,4)}&\!\!\footnotesize{(8,1,1,1,1)}\!\!&\!\footnotesize{$S(8,8,8,8)\dotsto S(2,4,8)$}\!&
    \footnotesize{\textbf{Excep}}&\footnotesize{\texttt{69}}\\
            \cline{2-7}
&\footnotesize{(2,\ldots,2)}&\footnotesize{(4,4,4)}&\!\!\footnotesize{(6,3,1,1,1)}\!\!&\!\footnotesize{$S(2,6,6,6)\dotsto S(2,4,6)$}\!&
    $\checkmark$&\footnotesize{\texttt{70}}\\ \cline{2-7}
&\footnotesize{(2,\ldots,2)}&\footnotesize{(4,4,4)}&\!\!\footnotesize{(6,2,2,1,1)}\!\!&\!\footnotesize{$S(3,3,6,6)\dotsto S(2,4,6)$}\!&
    $\checkmark$&\footnotesize{\texttt{71}}\\ \cline{2-7}
&\footnotesize{(2,\ldots,2)}&\footnotesize{(3,3,3,3)}&\footnotesize{(7,3,1,1)}&\!\footnotesize{$S(3,7,21,21)\dotsto S(2,3,21)$}\!&
    \footnotesize{\textbf{Excep}}&\footnotesize{\texttt{72}}\\
            \cline{2-7}
&\footnotesize{(2,\ldots,2)}&\footnotesize{(3,3,3,3)}&\footnotesize{(7,2,2,1)}&\!\footnotesize{$S(2,7,7,14)\dotsto S(2,3,14)$}\!&
    \footnotesize{\textbf{Excep}}&\footnotesize{\texttt{73}}\\
            \cline{2-7}
&\footnotesize{(2,\ldots,2)}&\footnotesize{(3,3,3,3)}&\footnotesize{(6,4,1,1)}&\!\footnotesize{$S(2,3,12,12)\dotsto S(2,3,12)$}\!&
    \footnotesize{\textbf{Excep}}&\footnotesize{\texttt{74}}\\
            \cline{2-7}
&\footnotesize{(2,\ldots,2)}&\footnotesize{(3,3,3,3)}&\footnotesize{(5,4,2,1)}&\!\footnotesize{$S(4,5,10,20)\dotsto S(2,3,20)$}\!&
    \footnotesize{\textbf{Excep}}&\footnotesize{\texttt{75}}\\
            \cline{2-7}
&\footnotesize{(2,\ldots,2)}&\footnotesize{(3,3,3,3)}&\footnotesize{(5,3,3,1)}&\!\footnotesize{$S(3,5,5,15)\dotsto S(2,3,15)$}\!&
    \footnotesize{\textbf{Excep}}&\footnotesize{\texttt{76}}\\
            \cline{2-7}
&\footnotesize{(2,\ldots,2)}&\footnotesize{(3,3,3,3)}&\footnotesize{(5,3,2,2)}&\!\footnotesize{$S(6,10,15,15)\dotsto S(2,3,30)$}\!&
    \footnotesize{\textbf{Excep}}&\footnotesize{\texttt{77}}\\
            \cline{2-7}
&\footnotesize{(2,\ldots,2)}&\footnotesize{(3,3,3,3)}&\footnotesize{(4,3,3,2)}&\!\footnotesize{$S(3,4,4,6)\dotsto S(2,3,12)$}\!&
    \footnotesize{\textbf{Excep}}&\footnotesize{\texttt{78}}\\
            \cline{2-7}
&\footnotesize{(2,\ldots,2)}&\footnotesize{(3,3,3,3)}&\footnotesize{(4,4,3,1)}&\!\footnotesize{$S(3,3,4,12)\dotsto S(2,3,12)$}\!&
    \footnotesize{\textbf{Excep}}&\footnotesize{\texttt{79}}\\
            \cline{2-7}
&\footnotesize{(3,3,3,3)}&\footnotesize{(3,3,3,3)}&\!\!\footnotesize{(4,4,1,1,1,1)}\!\!&\!\footnotesize{$S(4,4,4,4)\dotsto S(3,3,4)$}\!&
    $\checkmark$&\footnotesize{\texttt{80}}\\ \cline{2-7}
&\footnotesize{(2,\ldots,2)}&\footnotesize{(5,5,1,1)}&\footnotesize{(5,5,1,1)}&\!\footnotesize{$S(5,5,5,5)\dotsto S(2,5,5)$}\!&
    $\checkmark$&\footnotesize{\texttt{81}}\\ \cline{2-7}
&\footnotesize{(2,\ldots,2)}&\footnotesize{(4,4,2,2)}&\footnotesize{(5,5,1,1)}&\!\footnotesize{$S(2,2,5,5)\dotsto S(2,4,5)$}\!&
    $\checkmark$&\footnotesize{\texttt{82}}
\end{tabular}
\end{center}
\mycap{Continued from Table~\ref{hyp:S(4)-to-S(3):summary:13-47:tab}.\label{hyp:S(4)-to-S(3):summary:48-82:tab}}
\end{table}

\clearpage

\begin{table}
\begin{center}
\begin{tabular}{c|c|c|c|c|c|c}
$d$&$\Pi_1$&$\Pi_2$&$\Pi_3$&\footnotesize{Associated} $\Xtil\dotsto X$&\!\!\footnotesize{Realizable?}\!\!&$\#$\\ \hline\hline
\footnotesize{13}
&\footnotesize{(2,\ldots,2,1)}&\footnotesize{(3,3,3,3,1)}&\footnotesize{(8,4,1)}&\!\footnotesize{$S(2,2,3,8)\dotsto S(2,3,8)$}\!&
    $\checkmark$&\footnotesize{\texttt{83}}\\
            \cline{2-7}
&\footnotesize{(2,\ldots,2,1)}&\footnotesize{(3,3,3,3,1)}&\footnotesize{(11,1,1)}&\!\scriptsize{$S(2,3,11,11)\dotsto S(2,3,11)$}\!&
    $\checkmark$&\footnotesize{\texttt{84}}\\
            \cline{2-7}
&\footnotesize{(2,\ldots,2,1)}&\footnotesize{(3,3,3,3,1)}&\footnotesize{(10,2,1)}&\!\footnotesize{$S(2,3,5,10)\dotsto S(2,3,10)$}\!&
    $\checkmark$&\footnotesize{\texttt{85}}\\
            \cline{2-7}
&\footnotesize{(2,\ldots,2,1)}&\footnotesize{(3,3,3,3,1)}&\footnotesize{(9,3,1)}&\!\footnotesize{$S(2,3,3,9)\dotsto S(2,3,9)$}\!&
    $\checkmark$&\footnotesize{\texttt{86}}\\
            \hline

\footnotesize{14}
&\footnotesize{(2,\ldots,2)}&\footnotesize{(3,3,3,3,1,1)}&\footnotesize{(10,2,2)}&\!\footnotesize{$S(3,3,5,5)\dotsto S(2,3,10)$}\!&
    $\checkmark$&\footnotesize{\texttt{87}}\\
            \cline{2-7}
&\footnotesize{(2,\ldots,2)}&\footnotesize{(3,3,3,3,1,1)}&\footnotesize{(8,4,2)}&\!\footnotesize{$S(2,3,3,4)\dotsto S(2,3,8)$}\!&
    $\checkmark$&\footnotesize{\texttt{88}}\\
            \cline{2-7}
&\footnotesize{(2,\ldots,2)}&\footnotesize{(3,3,3,3,1,1)}&\footnotesize{(12,1,1)}&\!\scriptsize{$S(3,3,12,12)\dotsto S(2,3,12)$}\!&
    $\checkmark$&\footnotesize{\texttt{89}}\\
            \cline{2-7}
&\footnotesize{(2,\ldots,2)}&\footnotesize{(4,4,4,1,1)}&\footnotesize{(6,6,1,1)}&\!\footnotesize{$S(4,4,6,6)\dotsto S(2,4,6)$}\!&
    $\checkmark$&\footnotesize{\texttt{90}}\\ \cline{2-7}
&\footnotesize{(2,\ldots,2,1,1)}&\footnotesize{(3,3,3,3,1,1)}&\footnotesize{(7,7)}&\!\footnotesize{$S(2,2,3,3)\dotsto S(2,3,7)$}\!&
    $\checkmark$&\footnotesize{\texttt{91}}\\ \hline

\footnotesize{15}
&\footnotesize{(2,\ldots,2,1)}&\footnotesize{(3,3,3,3,3)}&\footnotesize{(12,1,1,1)}&\!\scriptsize{$S(2,12,12,12)\dotsto S(2,3,12)$}\!&
    $\checkmark$&\footnotesize{\texttt{92}}\\
            \cline{2-7}
&\footnotesize{(2,\ldots,2,1)}&\footnotesize{(3,3,3,3,3)}&\footnotesize{(8,4,2,1)}&\!\footnotesize{$S(2,2,4,8)\dotsto S(2,3,8)$}\!&
    $\checkmark$&\footnotesize{\texttt{93}}\\
            \cline{2-7}
&\footnotesize{(2,\ldots,2,1)}&\footnotesize{(3,3,3,3,3)}&\footnotesize{(10,2,2,1)}&\!\footnotesize{$S(2,5,5,10)\dotsto S(2,3,10)$}\!&
    $\checkmark$&\footnotesize{\texttt{94}}\\
            \cline{2-7}
&\footnotesize{(2,\ldots,2,1)}&\footnotesize{(4,4,4,1,1,1)}&\footnotesize{(5,5,5)}&\!\footnotesize{$S(2,4,4,4)\dotsto S(2,4,5)$}\!&
    $\checkmark$&\footnotesize{\texttt{95}}\\
            \cline{2-7}
&\!\!\footnotesize{(2,\ldots,2,1,1,1)}\!\!&\footnotesize{(3,3,3,3,3)}&\footnotesize{(7,7,1)}&\!\footnotesize{$S(2,2,2,7)\dotsto S(2,3,7)$}\!&
    $\checkmark$&\footnotesize{\texttt{96}}\\ \hline

\footnotesize{16}
&\footnotesize{(2,\ldots,2)}&\footnotesize{(3,3,3,3,3,1)}&\footnotesize{(10,2,2,2)}&\!\footnotesize{$S(3,5,5,5)\dotsto S(2,3,10)$}\!&
    \footnotesize{\textbf{Excep}}&\footnotesize{\texttt{97}}\\
            \cline{2-7}
&\footnotesize{(2,\ldots,2)}&\footnotesize{(3,3,3,3,3,1)}&\footnotesize{(8,4,2,2)}&\!\footnotesize{$S(2,3,4,4)\dotsto S(2,3,8)$}\!&
    \footnotesize{\textbf{Excep}}&\footnotesize{\texttt{98}}\\
            \cline{2-7}
&\footnotesize{(2,\ldots,2)}&\!\!\footnotesize{(3,3,3,3,1,1,1,1)}\!\!&\footnotesize{(8,8)}&\!\footnotesize{$S(3,3,3,3)\dotsto S(2,3,8)$}\!&
    $\checkmark$&\footnotesize{\texttt{99}}\\
            \cline{2-7}
&\footnotesize{(2,\ldots,2)}&\footnotesize{(3,3,3,3,3,1)}&\footnotesize{(13,1,1,1)}&\!\scriptsize{$S(3,13,13,13)\dotsto S(2,3,13)$}\!&
    $\checkmark$&\footnotesize{\texttt{100}}\\
            \cline{2-7}
&\footnotesize{(2,\ldots,2)}&\footnotesize{(3,3,3,3,3,1)}&\footnotesize{(12,2,1,1)}&\!\scriptsize{$S(3,6,12,12)\dotsto S(2,3,12)$}\!&
    $\checkmark$&\footnotesize{\texttt{101}}\\
            \cline{2-7}
&\footnotesize{(2,\ldots,2)}&\footnotesize{(3,3,3,3,3,1)}&\footnotesize{(9,3,3,1)}&\!\footnotesize{$S(3,3,3,9)\dotsto S(2,3,9)$}\!&
    $\checkmark$&\footnotesize{\texttt{102}}\\
            \cline{2-7}
&\footnotesize{(2,\ldots,2)}&\footnotesize{(4,4,4,4)}&\!\!\footnotesize{(6,6,1,1,1,1)}\!\!&\!\footnotesize{$S(6,6,6,6)\dotsto S(2,4,6)$}\!&
    $\checkmark$&\footnotesize{\texttt{103}}\\
            \cline{2-7}
&\footnotesize{(2,\ldots,2)}&\footnotesize{(4,4,4,2,1,1)}&\footnotesize{(5,5,5,1)}&\!\footnotesize{$S(2,4,4,5)\dotsto S(2,4,5)$}\!&
    $\checkmark$&\footnotesize{\texttt{104}}\\ \hline

\footnotesize{17}
&\footnotesize{(2,\ldots,2,1)}&\footnotesize{(3,3,3,3,3,1,1)}&\footnotesize{(8,8,1)}&\!\footnotesize{$S(2,3,3,8)\dotsto S(2,3,8)$}\!&
    $\checkmark$&\footnotesize{\texttt{105}}\\ \cline{2-7}
&\footnotesize{(2,\ldots,2,1)}&\footnotesize{(4,4,4,4,1)}&\footnotesize{(5,5,5,1,1)}&\!\footnotesize{$S(2,4,5,5)\dotsto S(2,4,5)$}\!&
    $\checkmark$&\footnotesize{\texttt{106}}\\ \hline

\footnotesize{18}
&\footnotesize{(2,\ldots,2,1,1)}&\footnotesize{(3,\ldots,3)}&\footnotesize{(8,8,1,1)}&\!\footnotesize{$S(2,2,8,8)\dotsto S(2,3,8)$}\!&
    $\checkmark$&\footnotesize{\texttt{107}}\\
            \cline{2-7}
&\footnotesize{(2,\ldots,2)}&\!\!\footnotesize{(3,3,3,3,3,1,1,1)}\!\!&\footnotesize{(8,8,2)}&\!\footnotesize{$S(3,3,3,4)\dotsto S(2,3,8)$}\!&
    $\checkmark$&\footnotesize{\texttt{108}}\\
            \cline{2-7}
&\footnotesize{(2,\ldots,2)}&\footnotesize{(3,\ldots,3)}&\footnotesize{(14,1,1,1,1)}&\!\scriptsize{$S(14,14,14,14)\dotsto S(2,3,14)$}\!&
    $\checkmark$&\footnotesize{\texttt{109}}\\
            \cline{2-7}
&\footnotesize{(2,\ldots,2)}&\footnotesize{(3,\ldots,3)}&\footnotesize{(12,2,2,1,1)}&\!\scriptsize{$S(6,6,12,12)\dotsto S(2,3,12)$}\!&
    $\checkmark$&\footnotesize{\texttt{110}}\\
            \cline{2-7}
&\footnotesize{(2,\ldots,2)}&\footnotesize{(3,\ldots,3)}&\footnotesize{(12,3,1,1,1)}&\!\scriptsize{$S(4,12,12,12)\dotsto S(2,3,12)$}\!&
    $\checkmark$&\footnotesize{\texttt{111}}\\
            \cline{2-7}
&\footnotesize{(2,\ldots,2)}&\footnotesize{(3,\ldots,3)}&\footnotesize{(10,5,1,1,1)}&\!\scriptsize{$S(2,10,10,10)\dotsto S(2,3,10)$}\!&
    $\checkmark$&\footnotesize{\texttt{112}}\\
            \cline{2-7}
&\footnotesize{(2,\ldots,2)}&\footnotesize{(3,\ldots,3)}&\footnotesize{(8,4,4,1,1)}&\!\footnotesize{$S(2,2,8,8)\dotsto S(2,3,8)$}\!&
    \footnotesize{\textbf{Excep}}&\footnotesize{\texttt{113}}\\
            \cline{2-7}
&\footnotesize{(2,\ldots,2)}&\footnotesize{(3,\ldots,3)}&\footnotesize{(8,4,2,2,2)}&\!\footnotesize{$S(2,4,4,4)\dotsto S(2,3,8)$}\!&
    \footnotesize{\textbf{Excep}}&\footnotesize{\texttt{114}}\\
            \cline{2-7}
&\footnotesize{(2,\ldots,2)}&\footnotesize{(3,\ldots,3)}&\footnotesize{(10,2,2,2,2)}&\!\footnotesize{$S(5,5,5,5)\dotsto S(2,3,10)$}\!&
    \footnotesize{\textbf{Excep}}&\footnotesize{\texttt{115}}\\
            \cline{2-7}
&\footnotesize{(2,\ldots,2)}&\footnotesize{(4,4,4,4,2)}&\!\!\footnotesize{(5,5,5,1,1,1)}\!\!&\!\footnotesize{$S(2,5,5,5)\dotsto S(2,4,5)$}\!&
    \footnotesize{\textbf{Excep}}&\footnotesize{\texttt{116}}

\end{tabular}
\end{center}
\mycap{Continued from Table~\ref{hyp:S(4)-to-S(3):summary:48-82:tab}.\label{hyp:S(4)-to-S(3):summary:83-116:tab}}
\end{table}

\clearpage

\begin{table}
\begin{center}
\begin{tabular}{c|c|c|c|c|c|c}
$d$&$\Pi_1$&$\Pi_2$&$\Pi_3$&\footnotesize{Associated} $\Xtil\dotsto X$&\!\!\footnotesize{Realizable?}\!\!&$\#$\\ \hline\hline
\footnotesize{19}
&\footnotesize{(2,\ldots,2,1)}&\footnotesize{(3,\ldots,3,1)}&\footnotesize{(8,8,2,1)}&\!\footnotesize{$S(2,3,4,8)\dotsto S(2,3,8)$}\!&
    $\checkmark$&\footnotesize{\texttt{117}}\\ \hline
\footnotesize{20}
&\footnotesize{(2,\ldots,2)}&\footnotesize{(3,\ldots,3,1,1)}&\footnotesize{(8,8,2,2)}&\!\footnotesize{$S(3,3,4,4)\dotsto S(2,3,8)$}\!&
    $\checkmark$&\footnotesize{\texttt{118}}\\
            \cline{2-7}
&\footnotesize{(2,\ldots,2)}&\footnotesize{(3,\ldots,3,1,1)}&\footnotesize{(9,9,1,1)}&\!\footnotesize{$S(3,3,9,9)\dotsto S(2,3,9)$}\!&
    $\checkmark$&\footnotesize{\texttt{119}}\\
            \cline{2-7}
&\footnotesize{(2,\ldots,2)}&\!\!\footnotesize{(4,4,4,4,1,1,1,1)}\!\!&\footnotesize{(5,5,5,5)}&\!\footnotesize{$S(4,4,4,4)\dotsto S(2,4,5)$}\!&
    $\checkmark$&\footnotesize{\texttt{120}}\\ \hline
\footnotesize{21}
&\footnotesize{(2,\ldots,2,1)}&\footnotesize{(3,\ldots,3)}&\footnotesize{(9,9,1,1,1)}&\!\footnotesize{$S(2,9,9,9)\dotsto S(2,3,9)$}\!&
    $\checkmark$&\footnotesize{\texttt{121}}\\ \cline{2-7}
&\footnotesize{(2,\ldots,2,1)}&\footnotesize{(3,\ldots,3)}&\footnotesize{(8,8,2,2,1)}&\!\footnotesize{$S(2,4,4,8)\dotsto S(2,3,8)$}\!&
    \footnotesize{\textbf{Excep}}&\footnotesize{\texttt{122}}\\
            \cline{2-7}
&\footnotesize{(2,\ldots,2,1)}&\footnotesize{(3,\ldots,3,1,1,1)}&\footnotesize{(7,7,7)}&\!\footnotesize{$S(2,3,3,3)\dotsto S(2,3,7)$}\!&
    $\checkmark$&\footnotesize{\texttt{123}}\\ \hline
\footnotesize{22}
&\!\!\footnotesize{(2,\ldots,2,1,1)}\!\!&\footnotesize{(3,\ldots,3,1)}&\footnotesize{(7,7,7,1)}&\!\footnotesize{$S(2,2,3,7)\dotsto S(2,3,7)$}\!&
    $\checkmark$&\footnotesize{\texttt{124}}\\
            \cline{2-7}
&\footnotesize{(2,\ldots,2)}&\footnotesize{(3,\ldots,3,1)}&\footnotesize{(8,8,4,1,1)}&\!\footnotesize{$S(2,3,8,8)\dotsto S(2,3,8)$}\!&
    \footnotesize{\textbf{Excep}}&\footnotesize{\texttt{125}}\\
            \cline{2-7}
&\footnotesize{(2,\ldots,2)}&\footnotesize{(3,\ldots,3,1)}&\footnotesize{(8,8,2,2,2)}&\!\footnotesize{$S(3,4,4,4)\dotsto S(2,3,8)$}\!&
    \footnotesize{\textbf{Excep}}&\footnotesize{\texttt{126}}\\
            \cline{2-7}
&\footnotesize{(2,\ldots,2)}&\footnotesize{(4,4,4,4,4,1,1)}&\footnotesize{(5,5,5,5,1,1)}&\!\footnotesize{$S(4,4,5,5)\dotsto S(2,4,5)$}\!&
    $\checkmark$&\footnotesize{\texttt{127}}\\
            \hline

\footnotesize{24}
&\footnotesize{(2,\ldots,2)}&\footnotesize{(3,\ldots,3)}&\footnotesize{(8,8,2,2,2,2)}&\!\footnotesize{$S(4,4,4,4)\dotsto S(2,3,8)$}\!&
    $\checkmark$&\footnotesize{\texttt{128}}\\
            \cline{2-7}
&\footnotesize{(2,\ldots,2)}&\footnotesize{(3,\ldots,3)}&\footnotesize{(8,8,4,2,1,1)}&\!\footnotesize{$S(2,4,8,8)\dotsto S(2,3,8)$}\!&
    $\checkmark$&\footnotesize{\texttt{129}}\\ \cline{2-7}
&\footnotesize{(2,\ldots,2)}&\footnotesize{(3,\ldots,3)}&\footnotesize{(10,10,1,1,1,1)}&\!\!\scriptsize{$S(10,10,10,10)\dotsto S(2,3,10)$}\!\!&
    $\checkmark$&\footnotesize{\texttt{130}}\\
            \cline{2-7}
&\footnotesize{(2,\ldots,2)}&\footnotesize{(3,\ldots,3)}&\footnotesize{(9,9,3,1,1,1)}&\!\footnotesize{$S(3,9,9,9)\dotsto S(2,3,9)$}\!&
    \footnotesize{\textbf{Excep}}&\footnotesize{\texttt{131}}\\
            \cline{2-7}
&\footnotesize{(2,\ldots,2)}&\footnotesize{(4,\ldots,4)}&\footnotesize{(5,5,5,5,1,1,1,1)}&\!\footnotesize{$S(5,5,5,5)\dotsto S(2,4,5)$}\!&
    $\checkmark$&\footnotesize{\texttt{132}}\\
            \hline

\footnotesize{26}
&\footnotesize{(2,\ldots,2)}&\footnotesize{(3,\ldots,3,1,1)}&\footnotesize{(8,8,8,1,1)}&\!\footnotesize{$S(3,3,8,8)\dotsto S(2,3,8)$}\!&
    $\checkmark$&\footnotesize{\texttt{133}}\\ \hline

\footnotesize{27}
&\footnotesize{(2,\ldots,2,1)}&\footnotesize{(3,\ldots,3)}&\footnotesize{(8,8,8,1,1,1)}&\!\footnotesize{$S(2,8,8,8)\dotsto S(2,3,8)$}\!&
    $\checkmark$&\footnotesize{\texttt{134}}\\ \hline

\footnotesize{28}
&\footnotesize{(2,\ldots,2)}&\footnotesize{(3,\ldots,3,1)}&\footnotesize{(8,8,8,2,1,1)}&\!\footnotesize{$S(3,4,8,8)\dotsto S(2,3,8)$}\!&
    $\checkmark$&\footnotesize{\texttt{135}}\\
        \cline{2-7}
&\footnotesize{(2,\ldots,2)}&\!\!\footnotesize{(3,\ldots,3,1,1,1,1)}\!\!&\footnotesize{(7,7,7,7)}&\!\footnotesize{$S(3,3,3,3)\dotsto S(2,3,7)$}\!&
    $\checkmark$&\footnotesize{\texttt{136}}\\
            \hline

\footnotesize{29}
&\footnotesize{(2,\ldots,2,1)}&\footnotesize{(3,\ldots,3,1,1)}&\footnotesize{(7,7,7,7,1)}&\!\footnotesize{$S(2,3,3,7)\dotsto S(2,3,7)$}\!&
    $\checkmark$&\footnotesize{\texttt{137}}\\
            \hline

\footnotesize{30}
&\!\!\footnotesize{(2,\ldots,2,1,1)}\!\!&\footnotesize{(3,\ldots,3)}&\footnotesize{(7,7,7,7,1,1)}&\!\footnotesize{$S(2,2,7,7)\dotsto S(2,3,7)$}\!&
    $\checkmark$&\footnotesize{\texttt{138}}\\
            \cline{2-7}
&\footnotesize{(2,\ldots,2)}&\footnotesize{(3,\ldots,3)}&\footnotesize{(8,8,8,2,2,1,1)}&\!\footnotesize{$S(4,4,8,8)\dotsto S(2,3,8)$}\!&
    $\checkmark$&\footnotesize{\texttt{139}}\\ \hline

\footnotesize{36}
&\footnotesize{(2,\ldots,2)}&\footnotesize{(3,\ldots,3)}&\!\!\footnotesize{(8,8,8,8,1,1,1,1)}\!\!&\!\footnotesize{$S(8,8,8,8)\dotsto S(2,3,8)$}\!&
    $\checkmark$&\footnotesize{\texttt{140}}\\
            \cline{2-7}
&\footnotesize{(2,\ldots,2)}&\footnotesize{(3,\ldots,3,1,1,1)}&\footnotesize{(7,7,7,7,7,1)}&\!\footnotesize{$S(3,3,3,7)\dotsto S(2,3,7)$}\!&
    $\checkmark$&\footnotesize{\texttt{141}}\\
            \hline
\footnotesize{37}
&\footnotesize{(2,\ldots,2,1)}&\footnotesize{(3,\ldots,3,1)}&\footnotesize{(7,7,7,7,7,1,1)}&\!\footnotesize{$S(2,3,7,7)\dotsto S(2,3,7)$}\!&
    $\checkmark$&\footnotesize{\texttt{142}}\\
            \hline
\footnotesize{44}
&\footnotesize{(2,\ldots,2)}&\footnotesize{(3,\ldots,3,1,1)}&\footnotesize{(7,\ldots,7,1,1)}&\!\footnotesize{$S(3,3,7,7)\dotsto S(2,3,7)$}\!&
    $\checkmark$&\footnotesize{\texttt{143}}
            \\ \hline
\footnotesize{45}
&\footnotesize{(2,\ldots,2,1)}&\footnotesize{(3,\ldots,3)}&\footnotesize{(7,\ldots,7,1,1,1)}&\!\footnotesize{$S(2,7,7,7)\dotsto S(2,3,7)$}\!&
    $\checkmark$&\footnotesize{\texttt{144}}
            \\ \hline
\footnotesize{52}
&\footnotesize{(2,\ldots,2)}&\footnotesize{(3,\ldots,3,1)}&\footnotesize{(7,\ldots,7,1,1,1)}&\!\footnotesize{$S(3,7,7,7)\dotsto S(2,3,7)$}\!&
    $\checkmark$&\footnotesize{\texttt{145}}
            \\ \hline
\footnotesize{60}
&\footnotesize{(2,\ldots,2)}&\footnotesize{(3,\ldots,3)}&\!\!\footnotesize{(7,\ldots,7,1,1,1,1)}\!\!&\!\footnotesize{$S(7,7,7,7)\dotsto S(2,3,7)$}\!&
    $\checkmark$&\footnotesize{\texttt{146}}

\end{tabular}
\end{center}
\mycap{Continued from Table~\ref{hyp:S(4)-to-S(3):summary:83-116:tab}.\label{hyp:S(4)-to-S(3):summary:117-146:tab}}
\end{table}

\clearpage

\begin{table}
\begin{center}
\begin{tabular}{c|c|c|c|c|c|c}
$d$&$\Pi_1$&$\Pi_2$&$\Pi_3$&\footnotesize{Associated} $\Xtil\dotsto X$&\!\!\footnotesize{Realizable?}\!\!&$\#$\\ \hline\hline
\footnotesize{4}
&\footnotesize{(3,1)}&\footnotesize{(4)}&\footnotesize{(4)}&\footnotesize{$T(3)\dotsto S(3,4,4)$}&
    $\checkmark$&\footnotesize{\texttt{147}}\\ \hline

\footnotesize{5}
&\footnotesize{(2,2,1)}&\footnotesize{(5)}&\footnotesize{(5)}&\footnotesize{$T(2)\dotsto S(2,5,5)$}&
    $\checkmark$&\footnotesize{\texttt{148}}\\ \hline

\footnotesize{6}
&\footnotesize{(2,2,2)}&\footnotesize{(5,1)}&\footnotesize{(6)}&\footnotesize{$T(5)\dotsto S(2,5,6)$}&
    $\checkmark$&\footnotesize{\texttt{149}}\\ \cline{2-7}
&\footnotesize{(2,2,2)}&\footnotesize{(4,2)}&\footnotesize{(6)}&\footnotesize{$T(2)\dotsto S(2,4,6)$}&
    $\checkmark$&\footnotesize{\texttt{150}}\\ \cline{2-7}
&\footnotesize{(3,3)}&\footnotesize{(3,3)}&\footnotesize{(5,1)}&\footnotesize{$T(5)\dotsto S(3,3,5)$}&
    $\checkmark$&\footnotesize{\texttt{151}}\\ \cline{2-7}
&\footnotesize{(3,3)}&\footnotesize{(3,3)}&\footnotesize{(4,2)}&\footnotesize{$T(2)\dotsto S(3,3,4)$}&
    \footnotesize{\textbf{Excep}}&\footnotesize{\texttt{152}}\\ \hline

\footnotesize{8}
&\footnotesize{(2,2,2,2)}&\footnotesize{(4,4)}&\footnotesize{(7,1)}&\footnotesize{$T(7)\dotsto S(2,4,7)$}&
    $\checkmark$&\footnotesize{\texttt{153}}\\ \cline{2-7}
&\footnotesize{(2,2,2,2)}&\footnotesize{(4,4)}&\footnotesize{(6,2)}&\footnotesize{$T(3)\dotsto S(2,4,6)$}&
    $\checkmark$&\footnotesize{\texttt{154}}\\ \hline

\footnotesize{9}
&\footnotesize{(2,2,2,2,1)}&\footnotesize{(3,3,3)}&\footnotesize{(9)}&\footnotesize{$T(2)\dotsto S(2,3,9)$}&
    $\checkmark$&\footnotesize{\texttt{155}}\\ \cline{2-7}
&\footnotesize{(3,3,3)}&\footnotesize{(3,3,3)}&\footnotesize{(4,4,1)}&\footnotesize{$T(4)\dotsto S(3,3,4)$}&
    $\checkmark$&\footnotesize{\texttt{156}}\\ \hline

\footnotesize{10}
&\footnotesize{(2,2,2,2,2)}&\footnotesize{(4,4,2)}&\footnotesize{(5,5)}&\footnotesize{$T(2)\dotsto S(2,4,5)$}&
    $\checkmark$&\footnotesize{\texttt{157}}\\ \cline{2-7}
&\footnotesize{(2,2,2,2,2)}&\footnotesize{(3,3,3,1)}&\footnotesize{(10)}&\footnotesize{$T(3)\dotsto S(2,3,10)$}&
    $\checkmark$&\footnotesize{\texttt{158}}\\ \hline

\footnotesize{12}
&\footnotesize{(2,\ldots,2)}&\footnotesize{(3,3,3,3)}&\footnotesize{(11,1)}&\footnotesize{$T(11)\dotsto S(2,3,11)$}&
    $\checkmark$&\footnotesize{\texttt{159}}\\ \cline{2-7}
&\footnotesize{(2,\ldots,2)}&\footnotesize{(3,3,3,3)}&\footnotesize{(10,2)}&\footnotesize{$T(5)\dotsto S(2,3,10)$}&
    $\checkmark$&\footnotesize{\texttt{160}}\\ \cline{2-7}
&\footnotesize{(2,\ldots,2)}&\footnotesize{(3,3,3,3)}&\footnotesize{(9,3)}&\footnotesize{$T(3)\dotsto S(2,3,9)$}&
    $\checkmark$&\footnotesize{\texttt{161}}\\ \cline{2-7}
&\footnotesize{(2,\ldots,2)}&\footnotesize{(3,3,3,3)}&\footnotesize{(8,4)}&\footnotesize{$T(2)\dotsto S(2,3,8)$}&
    $\checkmark$&\footnotesize{\texttt{162}}\\ \hline

\footnotesize{16}
&\footnotesize{(2,\ldots,2)}&\footnotesize{(3,3,3,3,3,1)}&\footnotesize{(8,8)}&\footnotesize{$T(3)\dotsto S(2,3,8)$}&
    \footnotesize{\textbf{Excep}}&\footnotesize{\texttt{163}}\\ \cline{2-7}
&\footnotesize{(2,\ldots,2)}&\footnotesize{(4,4,4,4)}&\footnotesize{(5,5,5,1)}&\footnotesize{$T(5)\dotsto S(2,4,5)$}&
    \footnotesize{\textbf{Excep}}&\footnotesize{\texttt{164}}\\ \hline

\footnotesize{18}
&\footnotesize{(2,\ldots,2)}&\footnotesize{(3,\ldots,3)}&\footnotesize{(8,8,2)}&\footnotesize{$T(4)\dotsto S(2,3,8)$}&
    $\checkmark$&\footnotesize{\texttt{165}}\\ \hline

\footnotesize{21}
&\footnotesize{(2,\ldots,2,1)}&\footnotesize{(3,\ldots,3)}&\footnotesize{(7,7,7)}&\footnotesize{$T(2)\dotsto S(2,3,7)$}&
    \footnotesize{\textbf{Excep}}&\footnotesize{\texttt{166}}\\ \hline

\footnotesize{28}
&\footnotesize{(2,\ldots,2)}&\footnotesize{(3,\ldots,3,1)}&\footnotesize{(7,7,7,7)}&\footnotesize{$T(3)\dotsto S(2,3,7)$}&
    $\checkmark$&\footnotesize{\texttt{167}}\\ \hline

\footnotesize{36}
&\footnotesize{(2,\ldots,2)}&\footnotesize{(3,\ldots,3)}&\footnotesize{(7,7,7,7,7,1)}&\footnotesize{$T(7)\dotsto S(2,3,7)$}&
    \footnotesize{\textbf{Excep}}&\footnotesize{\texttt{168}}

\end{tabular}
\end{center}
\mycap{Candidates $T\argdotstobis{d:1}{(\Pi_1,\Pi_2,\Pi_3)}S$
with associated hyperbolic $T(\alpha)\dotsto S(p,q,r)$.\label{hyp:T(1)-to-S(3):summary:tab}}
\end{table}

\clearpage

\section{Enumeration of relevant candidate covers}
\label{enum:sec}

In this section we establish the following two results:

\begin{thm}\label{hyp:S(4)-to-S(3):enum:thm}
The candidate surface branched covers with associated hyperbolic orbifold candidate
$S(\alpha,\beta,\gamma,\delta)\dotsto S(p,q,r)$
are precisely the $146$ items listed in
Tables~\ref{hyp:S(4)-to-S(3):summary:1-12:tab} to~\ref{hyp:S(4)-to-S(3):summary:117-146:tab}.
\end{thm}

\begin{thm}\label{hyp:T(1)-to-S(3):enum:thm}
The candidate surface branched covers with associated hyperbolic orbifold candidate
$T(\alpha)\dotsto S(p,q,r)$ are precisely the $22$ items listed in
Table~\ref{hyp:T(1)-to-S(3):summary:tab}.
\end{thm}

\dimo{hyp:S(4)-to-S(3):enum:thm}
Given a degree $d$ and three partitions $\Pi_1,\Pi_2,\Pi_3$ of $d$ we recall that $\ell(\Pi_i)$ denotes
the length of $\Pi_i$, and we note that
the Riemann-Hurwitz formula~(\ref{RH:general:eq}) reads
\begin{equation}
\ell(\Pi_1)+\ell(\Pi_2)+\ell(\Pi_3)=d+2\label{RH:S-to-S(3):eq}
\end{equation}
in this case, because $\Sigmatil=\Sigma$ is the sphere $S$. In addition we define
$c(\Pi_i)$ as the number of entries in
$\Pi_i$ which are different from ${\rm l.c.m.}(\Pi_i)$.
We must then find those $d$ and $\Pi_1,\Pi_2,\Pi_3$
satisfying~(\ref{RH:S-to-S(3):eq}), the relation
\begin{equation}
c(\Pi_1)+c(\Pi_2)+c(\Pi_3)=4\label{c=4:eq}
\end{equation}
and such that for the associated candidate
$\Xtil\dotsto X$  one has that $X$ (or, equivalently, $\Xtil$) is hyperbolic.
We begin with the following:

\begin{prop}\label{hyp:S(4)-to-S(3):d_leq_12:prop}
The relevant degrees $d$ and partitions
$\Pi_1,\Pi_2,\Pi_3$ with $d\leqslant 12$ are the $82$ items listed in Tables~\ref{hyp:S(4)-to-S(3):summary:1-12:tab}
to~\ref{hyp:S(4)-to-S(3):summary:48-82:tab}.
\end{prop}

\begin{proof}
For each $d$ between $2$ and $12$ we must:
\begin{itemize}
\item[(\textbf{a})] List all the partitions $\Pi$ of $d$ such that $c(\Pi)\leqslant 4$, excluding $(1,\ldots,1)$;
\item[(\textbf{b})] Select an unordered
triple $\Pi_1,\Pi_2,\Pi_3$ meeting conditions~(\ref{RH:S-to-S(3):eq}) and~(\ref{c=4:eq});
\item[(\textbf{c})] Discard the triples such that the sum of reciprocals of ${\rm l.c.m.}(\Pi_i)$ for $i=1,2,3$ is greater than or equal to $1$.
\end{itemize}
Note that excluding $(1,\ldots,1)$ in (\textbf{a}) guarantees that in the orbifold cover $\Xtil\to X$
associated to $S\argdotstoter{d:1}{(\Pi_1,\Pi_2,\Pi_3)}S$ one has $X=S(p,q,r)$. Then (\textbf{c}) reads
$\frac1p+\frac1q+\frac1r<1$ and means that $X$ is hyperbolic.

Achieving tasks (\textbf{a}), (\textbf{b}), and (\textbf{c}) is a matter that
only requires a little time and care, and that
can also safely be carried out by computer. As an only example, we make the argument
explicit for $d=10$, addressing the reader to~\cite{hyp1:proofs} for the other cases.
We first show in Table~\ref{good:partitions:d=10:tab} the $28$
partitions $\Pi$ of $d=10$ with $c(\Pi)\leqslant 4$.

\begin{table}
\begin{center}
\begin{tabular}{c||c|c|c|c|c|c}\footnotesize{
$\Pi$       }&\footnotesize{ (10)      }&\footnotesize{ (9,1)     }&\footnotesize{ (8,2)      }&\footnotesize{ (7,3)      }&\footnotesize{ (6,4)      }&\footnotesize{ (5,5) }\\ \hline{
$\ell$      }&\footnotesize{ 1         }&\footnotesize{ 2         }&\footnotesize{ 2          }&\footnotesize{  2         }&\footnotesize{  2         }&\footnotesize{  2 }\\ \hline{
$c$         }&\footnotesize{ 0         }&\footnotesize{ 1         }&\footnotesize{ 1          }&\footnotesize{  2         }&\footnotesize{  2         }&\footnotesize{  0 }\\ \hline\hline{

$\Pi$       }&\footnotesize{ (8,1,1)   }&\footnotesize{ (7,2,1)   }&\footnotesize{ (6,3,1)    }&\footnotesize{ (6,2,2)    }&\footnotesize{ (5,4,1)    }&\footnotesize{ (5,3,2) }\\ \hline{
$\ell$      }&\footnotesize{  3        }&\footnotesize{ 3         }&\footnotesize{ 3          }&\footnotesize{  3         }&\footnotesize{  3         }&\footnotesize{  3 }\\ \hline{
$c$         }&\footnotesize{  2        }&\footnotesize{ 3         }&\footnotesize{ 2          }&\footnotesize{  2         }&\footnotesize{  3         }&\footnotesize{  3 }\\ \hline\hline{

$\Pi$       }&\footnotesize{ (4,4,2)   }&\footnotesize{ (4,3,3)   }&\footnotesize{ (7,1,1,1) }&\footnotesize{ (6,2,1,1) }&\footnotesize{ (5,3,1,1)  }&\footnotesize{ (5,2,2,1)  }\\ \hline{
$\ell$      }&\footnotesize{  3        }&\footnotesize{ 3         }&\footnotesize{  4        }&\footnotesize{ 4         }&\footnotesize{ 4          }&\footnotesize{  4         }\\ \hline{
$c$         }&\footnotesize{  1        }&\footnotesize{ 3         }&\footnotesize{  3        }&\footnotesize{ 3         }&\footnotesize{ 4          }&\footnotesize{  4         }\\ \hline\hline{

$\Pi$       }&\footnotesize{ (4,4,1,1)  }&\footnotesize{ (4,3,2,1) }&\footnotesize{ (4,2,2,2)  }&\footnotesize{ (3,3,3,1)  }&\footnotesize{ (3,3,2,2)  }&\footnotesize{ (6,1,1,1,1) }\\ \hline{
$\ell$      }&\footnotesize{  4         }&\footnotesize{  4 }       &\footnotesize{ 4          }&\footnotesize{  4         }&\footnotesize{  4         }&\footnotesize{  5 }\\ \hline{
$c$         }&\footnotesize{  2         }&\footnotesize{  4 }       &\footnotesize{ 3          }&\footnotesize{  1         }&\footnotesize{  4         }&\footnotesize{  4 }\\ \hline\hline{

$\Pi$       }&\footnotesize{(4,2,2,1,1)}&\footnotesize{(2,\ldots,2)}&\footnotesize{ (3,3,1,1,1,1)}&\footnotesize{ (2,2,2,2,1,1)    }&\footnotesize{ (2,2,2,1,1,1,1)  }&\footnotesize{ }\\ \hline{
$\ell$      }&\footnotesize{  5        }&\footnotesize{ 5         } &\footnotesize{ 6             }&\footnotesize{  6               }&\footnotesize{  7               }&\footnotesize{ }\\ \hline{
$c$         }&\footnotesize{  4        }&\footnotesize{ 0         } &\footnotesize{ 4             }&\footnotesize{  2               }&\footnotesize{  4               }&
\end{tabular}
\end{center}
\mycap{The partitions $\Pi$ of $d=10$ with $c(\Pi)\leqslant 4$\label{good:partitions:d=10:tab}}
\end{table}

Analyzing these partitions we see that the possible values of the pairs $(c,\ell)$ and the numbers of partitions giving
each of them are those in Table~\ref{pairs:(c,l):d=10:tab}.

\begin{table}
\begin{center}
\begin{tabular}
{l      ||l c c|c c c|c c c c|c c|c c c c}
$c$      &0& & &1& & &2& & & &3& &4& & & \\ \hline
$\ell$   &1&2&5&2&3&4&2&3&4&6&3&4&4&5&6&7\\ \hline\hline
$\#$     &1&1&1&2&1&1&2&3&1&1&4&3&4&2&1&1
\end{tabular}
\end{center}
\mycap{Values of $(c,\ell)$ for the partitions $\Pi$ of $d=10$ with $c(\Pi)\leqslant 4$, and
numbers of $\Pi$'s giving each $(c,\ell)$.\label{pairs:(c,l):d=10:tab}}
\end{table}

\begin{table}
\begin{center}
\begin{tabular}
{c              | c                 | c                 ||c}
$(c_1,\ell_1)$  & $(c_2,\ell_2)$    & $(c_3,\ell_3)$    & $\#$  \\ \hline\hline
(0,5)           & (0,2)             & (4,5)             & 2     \\ \hline
(0,5)           & (0,1)             & (4,6)             & 1     \\ \hline
(0,5)           & (1,4)             & (3,3)             & 4     \\ \hline
(0,5)           & (1,3)             & (3,4)             & 3     \\ \hline
(0,5)           & (2,4)             & (2,3)             & 3     \\ \hline
(0,2)           & (2,4)             & (2,6)             & 1     \\ \hline
(1,4)           & (1,4)             & (2,4)             & 1     \\ \hline
(1,2)           & (1,4)             & (2,6)             & 2     \\ \hline
(1,3)           & (1,3)             & (2,6)             & 1
\end{tabular}
\end{center}
\mycap{Triples $(c,\ell)$ summing up to $(4,12)$, and numbers of different choices for
the corresponding partitions.\label{good:triples:d=10:tab}}
\end{table}

\begin{table}
\begin{center}
\begin{tabular}{c|c|c|c|c}
$\Pi_1$      & $\Pi_2$ & $ \Pi_3$ & Associated cover & Geometry \\ \hline\hline
(2,2,2,2,2) & (5,5) & (6,1,1,1,1)  & $S(6,6,6,6)\dotsto S(2,5,6)$ & $\matH$ \\     
(2,2,2,2,2) & (5,5) & (4,2,2,1,1)  & $S(2,2,4,4)\dotsto S(2,4,5)$ & $\matH$ \\ \hline    

(2,2,2,2,2) & (10) & (3,3,1,1,1,1) & $S(3,3,3,3)\dotsto S(2,3,10)$ & $\matH$ \\ \hline     

(2,2,2,2,2) & (3,3,3,1) & (7,2,1)  & $S(2,3,7,14)\dotsto S(2,3,14)$ & $\matH$ \\   
(2,2,2,2,2) & (3,3,3,1) & (5,4,1)  & $S(3,4,5,20)\dotsto S(2,3,20)$ & $\matH$ \\   
(2,2,2,2,2) & (3,3,3,1) & (5,3,2)  & $S(3,6,10,15)\dotsto S(2,3,30)$ & $\matH$ \\  
(2,2,2,2,2) & (3,3,3,1) & (4,3,3)  & $S(3,3,4,4)\dotsto S(2,3,12)$ & $\matH$ \\ \hline  

(2,2,2,2,2) & (4,4,2) & (7,1,1,1)  & $S(2,7,7,7)\dotsto S(2,4,7)$ & $\matH$ \\    
(2,2,2,2,2) & (4,4,2) & (6,2,1,1)  & $S(2,3,6,6)\dotsto S(2,4,6)$ & $\matH$ \\    
(2,2,2,2,2) & (4,4,2) & (4,2,2,2)  & $S(2,2,2,2)\dotsto S(2,4,4)$ & $\matE$ \\ \hline    

(2,2,2,2,2) & (4,4,1,1) & (8,1,1)  & $S(4,4,8,8)\dotsto S(2,4,8)$ & $\matH$ \\     
(2,2,2,2,2) & (4,4,1,1) & (6,3,1)  & $S(2,4,4,6)\dotsto S(2,4,6)$ & $\matH$ \\     
(2,2,2,2,2) & (4,4,1,1) & (6,2,2)  & $S(3,3,4,4)\dotsto S(2,4,6)$ & $\matH$ \\ \hline    

(5,5) & (4,4,1,1) & (2,2,2,2,1,1)   & $S(2,2,4,4)\dotsto S(2,4,5)$ & $\matH$ \\ \hline    

(3,3,3,1) & (3,3,3,1) & (4,4,1,1)   & $S(3,3,4,4)\dotsto S(3,3,4)$ & $\matH$ \\ \hline    

(8,2) & (3,3,3,1) & (2,2,2,2,1,1)    & $S(2,2,3,4)\dotsto S(2,3,8)$ & $\matH$ \\     
(9,1) & (3,3,3,1) & (2,2,2,2,1,1)    & $S(2,2,3,9)\dotsto S(2,3,9)$ & $\matH$ \\ \hline    

(4,4,2) & (4,4,2) & (2,2,2,2,1,1)  & $S(2,2,2,2)\dotsto S(2,4,4)$ & $\matE$     
\end{tabular}
\end{center}
\mycap{Partitions of $d=10$ giving rise to candidate orbifold covers of the form
$S(\alpha,\beta,\gamma,\delta)\dotsto S(p,q,r)$ and the corresponding geometries.\label{all:d=10:candidates:tab}}
\end{table}

To achieve task (\textbf{b}) we must now select all possible unordered triples of partitions such that the corresponding
$(c,\ell)$'s sum up to $(4,12)$, which can be done in the ways described in Table~\ref{all:d=10:candidates:tab}.
The table also contains the type geometry of $X$ and $\Xtil$ in the corresponding candidate orbifold
cover. Task (\textbf{c}) corresponds to discarding non-$\matH$ geometries, after which we get the
16 items with numbers \texttt{43} through \texttt{58} in Tables~\ref{hyp:S(4)-to-S(3):summary:13-47:tab}
and~\ref{hyp:S(4)-to-S(3):summary:48-82:tab}.
\end{proof}

We now turn to the following:

\begin{prop}\label{hyp:S(4)-to-S(3):d_geq_13:prop}
The degrees $d$ and partitions $\Pi_1,\Pi_2,\Pi_3$ relevant to Theorem~\ref{hyp:S(4)-to-S(3):enum:thm}
with $d\geqslant 13$ are the $64$ items listed in Tables~\ref{hyp:S(4)-to-S(3):summary:83-116:tab}
and~\ref{hyp:S(4)-to-S(3):summary:117-146:tab}.
\end{prop}

\begin{proof}
In this case we proceed in reverse order, from $\Xtil\dotsto X$ to $d$ and the partitions.
Namely we first analyze which hyperbolic candidate orbifold covers
$\Xtil\dotsto^{d:1} X$ exist with $d\geqslant 13$, and in the course of this analysis
we determine for what choices of $d$ and $\Pi_1,\Pi_2,\Pi_3$ these candidates can arise.

Let us then consider a hyperbolic candidate $\Xtil\dotsto^{d:1} X$ with $d\geqslant 13$,
$\Xtil=S(\alpha,\beta,\gamma,\delta)$ and $X=S(p,q,r)$.
Let us always assume that $\alpha\leqslant\beta\leqslant\gamma\leqslant\delta$ and $p\leqslant q\leqslant r$. Since
$0<-\chiorb(\Xtil)=2-\left(\frac1\alpha+\frac1\beta+\frac1\gamma+\frac1\delta\right)<2$
and $\chiorb(\Xtil)=d\myprod\chiorb (X)$, we deduce that
\begin{center}
$0<-\chiorb(X)=1-\left(\frac1p+\frac1q+\frac1r\right)<\frac2{13}
\quad\Rightarrow\quad
\frac{11}{13}<\frac1p+\frac1q+\frac1r<1$.
\end{center}
Assuming $p\leqslant q\leqslant r$ it is now very easy to check that
the last inequality is satisfied only for $(p,q,r)$ as follows:
\begin{itemize}
\item [(\textbf{A})] $(2,3,r)$ with $7\leqslant r\leqslant77$;
\item [(\textbf{B})] $(2,4,r)$ with $5\leqslant r\leqslant10$;
\item [(\textbf{C})] $(2,5,5)$, $(2,5,6)$, $(3,3,4)$ or $(3,3,5)$.
\end{itemize}

We now remark that:
\begin{itemize}
\item[(\textbf{I})] Each of $\alpha,\beta,\gamma,\delta$ must be a divisor of one of $p,q,r$;
\item[(\textbf{II})] $d=\frac{\chiorb(\Xtil)}{\chiorb(X)}$ must be an integer.
\end{itemize}

Let us now establish the following auxiliary result:

\begin{lem}\label{dmax:lem}
Let $S(\alpha,\beta,\gamma,\delta)\dotsto^{d:1} S(p,q,r)$ be a hyperbolic candidate orbifold cover.
Set $d_{\max}(p,q,r)=\frac{2-\frac4r}{1-\frac1p-\frac1q-\frac1r}$.
Then $d\leqslant d_{\max}(p,q,r)$.
\end{lem}

\begin{proof}
By (\textbf{I}) and the condition $p\leqslant q\leqslant r$ we have
$\alpha,\beta,\gamma,\delta\leqslant r$, whence
$-\chiorb(S(\alpha,\beta,\gamma,\delta))=2-\frac1\alpha-\frac1\beta-\frac1\gamma-\frac1\delta\leqslant2-\frac4r$
and the conclusion follows from~(\ref{RH:orb:eq}), because $-\chiorb(S(p,q,r))=1-\frac1p-\frac1q-\frac1r$.
\end{proof}

Getting back to the cases (\textbf{A}), (\textbf{B}), and (\textbf{C}) that we must consider, we start
from the last one and note that
\begin{center}
$d_{\max}(2,5,5)=12,\ \
d_{\max}(2,5,6)=10,\ \  d_{\max}(3,3,4)=12,\ \ d_{\max}(3,3,5)=9$.\end{center}
Since the assumption $d\geqslant 13$ is in force, we conclude that case (\textbf{C}) does not yield relevant candidates.
Turning to case (\textbf{B}), we have
\begin{center}
$d_{\max}(2,4,5)=24,\quad d_{\max}(2,4,6)=16,\quad d_{\max}(2,4,7)=13.\overline{3},$\\
$\quad d_{\max}(2,4,8)=12,\quad d_{\max}(2,4,9)=11.2,\quad d_{\max}(2,4,10)=10.\overline{6}$.
\end{center}
Therefore, the cases $r=8$, $r=9$ and $r=10$ do not yield relevant candidates. The case $r=7$, namely $X=S(2,4,7)$, also does not,
because for $\Xtil=S(7,7,7,7)$ we have that $\frac{\chiorb(\Xtil)}{\chiorb(X)}=13.\overline{3}$, which violates (\textbf{II}),
and for the next biggest possible $-\chiorb(\Xtil)$ in view of (\textbf{I}),
corresponding to $\Xtil=S(4,7,7,7)$, we have $\frac{\chiorb(\Xtil)}{\chiorb(X)}=12.\overline{3}$, which again violates (\textbf{II}).
Let now analyze the case $X=S(2,4,6)$. Picking $\alpha,\beta,\gamma,\delta\in\{2,3,4,6\}$,
as imposed by (\textbf{I}), and discarding the cases where
$\frac{\chiorb(\Xtil)}{\chiorb(X)}\leqslant 12$, we find the possible $\Xtil$'s of Table~\ref{Xtil:for:X=S(2,4,6):tab}.
\begin{table}
\begin{center}
\begin{tabular}{c||c|c|c|c|c|c}
$\Xtil$ &
   \footnotesize{$S(6,6,6,6)$} &
   \footnotesize{$S(4,6,6,6)$} &
   \footnotesize{$S(3,6,6,6)$} &
   \footnotesize{$S(4,4,6,6)$} &
   \footnotesize{$S(3,4,6,6)$} &
   \footnotesize{$S(4,4,4,6)$} \\ \hline
$d$ &
    16  &   15  &   14  &   14  &   13  &   13
\end{tabular}
\end{center}
\mycap{The relevant $\Xtil$ for $X=S(2,4,6)$
and the corresponding $d=\frac{\chiorb(\Xtil)}{\chiorb(X)}$.\label{Xtil:for:X=S(2,4,6):tab}}
\end{table}
To conclude with the case $X=S(2,4,6)$ we must now discuss for which
$\Xtil$ as in Table~\ref{Xtil:for:X=S(2,4,6):tab} there actually exist
partitions of $d$ inducing a candidate $\Xtil\to X$. We do this in full detail
to give the reader a taste of the arguments one can use to this end. In similar cases below
we will omit all details, addressing to~\cite{hyp1:proofs}.

\begin{prop}
The only candidate covers with $X=S(2,4,6)$ and $\Xtil$ as in Table~\ref{Xtil:for:X=S(2,4,6):tab}
are those described in items \emph{\texttt{90}} and \emph{\texttt{103}} in Table~\ref{hyp:S(4)-to-S(3):summary:83-116:tab}.
\end{prop}

\begin{proof}
For $\Xtil=S(6,6,6,6)$ the partition of $d=16$ corresponding
to the cone point of order $6$ in $X$ must include four $1$'s, which
already give all four cone points of $\Xtil$, so the partition must
be $(6,6,1,1,1,1)$ and the other two must be $(2,\ldots,2)$ and $(4,4,4,4)$, whence
item \texttt{103} in Table~\ref{hyp:S(4)-to-S(3):summary:83-116:tab}.
For $\Xtil=S(4,6,6,6)$, $\Xtil=S(3,4,6,6)$ and $\Xtil=S(4,4,4,6)$ the fact that
$\Xtil$ has no cone point of order $2$ implies that the partition of $d$
corresponding to the cone point of order $2$ in $X$ consists of $2$'s only,
which is impossible because $d$ is odd. A similar argument
shows that $\Xtil=S(3,6,6,6)$ does not give a candidate cover:
since $\Xtil$ has no cone point of order $2$ or $4$, the partition
corresponding to the cone point of order $4$ in $X$ must consist of $4$'s only,
which is impossible because $d=14$ is not a multiple of $4$.
Finally, let us consider $\Xtil=S(4,4,6,6)$; these cone orders tell us
that in both partitions corresponding to the cone
points of orders $4$ and $6$ in $X$ there must be two $1$'s, but
then the only possibility is $(2,\ldots,2)$, $(4,4,4,1,1)$ and $(6,6,1,1)$,
which gives item \texttt{90} in Table~\ref{hyp:S(4)-to-S(3):summary:83-116:tab}.
\end{proof}

To conclude case (\textbf{B}) we must deal with $X=S(2,4,5)$.
This is done in very much the same way as for $X=S(2,4,6)$ and leads to
items \texttt{95}, \texttt{104}, \texttt{106}, \texttt{116}, \texttt{120}, \texttt{127}, and \texttt{132}
in Tables~\ref{hyp:S(4)-to-S(3):summary:83-116:tab} and~\ref{hyp:S(4)-to-S(3):summary:117-146:tab},
see~\cite{hyp1:proofs}.

\bigskip

Finally, we examine case (\textbf{A}), where $X=S(2,3,r)$ and $7\leqslant r\leqslant 77$.
We first note that for the function $d_{\max}$ introduced
in Lemma~\ref{dmax:lem} we have $d_{\max}(2,3,r)=12\cdot \frac{r-2}{r-6}$. Imposing
$d\leqslant\left\lfloor d_{\max}(2,3,r)\right\rfloor$ we then easily get:

\begin{lem}\label{floor:dmax:lem}
Let $S(\alpha,\beta,\gamma,\delta)\dotsto^{d:1} S(2,3,r)$ be a candidate
orbifold cover. Then, depending on the value of $r$, the degree $d$ satisfies
the upper bound described in Table~\ref{floor:dmax:tab}.
\begin{table}
\begin{center}
\begin{tabular}
{l||
    c|
        c|
            c|
                c|
                    c}

If \ldots &
    $r=7$   &
        $r=8$   &
            $r=9$   &
                $r=10$  &
                    $r=11$ \\ \hline

then $d\leqslant$ &
    60      &
        36      &
            28      &
                24      &
                    21 \\ \hline\hline
If \ldots &
    $r=12$ &
        $13\leqslant r\leqslant 14$ &
            $r=15$ &
                $16\leqslant r\leqslant 18$ &
                    $19\leqslant r\leqslant 22$ \\ \hline
then $d\leqslant$ &
    20 &
        18 &
            17 &
                16 &
                    15 \\ \hline\hline
If \ldots &
    $23\leqslant r\leqslant 30$ &
        $31\leqslant r\leqslant 54$ &
            $r\geqslant 55$ &
                &
                    \\ \hline
then $d\leqslant$ &
    14 &
        13 &
            12 &
                &
\end{tabular}
\end{center}
\mycap{Upper bounds on $d$ depending on the values of $r$.\label{floor:dmax:tab}}
\end{table}
\end{lem}

Since our aim is to list the relevant candidate covers with $d\geqslant 13$
we see in particular that we can restrict to
$7\leqslant r\leqslant 54$.
This leaves however several cases to consider, to reduce which
we establish the following:

\begin{lem}\label{dcong:lem}
Let $S(\alpha,\beta,\gamma,\delta)\dotsto^{d:1} S(2,3,r)$ be a candidate
orbifold cover. Set
$$
\begin{array}{cc}
d_{\max}^{\,\equiv1\,(2)}(r)=9\cdot\frac{r-2}{r-6},&
d_{\max}^{\,\equiv1\,(3)}(r)=2\cdot\frac{5r-9}{r-6},\\
d_{\max}^{\,\equiv2\,(3)}(r)=4\cdot\frac{2r-3}{r-6},&
d_{\max}^{\,\equiv1\,(6)}(r)=\frac{7r-12}{r-6}.
\end{array}
$$
If $d\equiv k\ (\textrm{mod}\ n)$ then $d\leqslant d_{\max}^{\,\equiv k\,(n)}(r)$.
\end{lem}

\begin{proof}
If $d\equiv 1\ (\textrm{mod}\ 2)$ then the partition of $d$ corresponding to
the cone point of order $2$ in $S(2,3,r)$ must have at least one entry equal to $1$,
so $\alpha=2$, whence $-\chiorb(S(\alpha,\beta,\gamma,\delta))\leqslant 2-\frac12-\frac3r$;
therefore, $d\leqslant\frac{2-\frac12-\frac3r}{1-\frac12-\frac13-\frac1r}=d_{\max}^{\,\equiv1\,(2)}(r)$.
The other cases are treated in a similar way.
\end{proof}

\begin{lem}\label{dcong:bounds:lem}
Let $S(\alpha,\beta,\gamma,\delta)\dotsto^{d:1} S(2,3,r)$ be a candidate
orbifold cover with $d\geqslant 13$. Then, depending on the congruence class of $d$ modulo $6$,
the cone order $r$ satisfies the upper bounds described in Table~\ref{cong:bounds:tab}.
\begin{table}
\begin{center}
\begin{tabular}
{l||                                c|      c|      c|      c|      c|      c      }
If $d\equiv \ldots\ (\textrm{mod}\ 6)$ &   0   &   1   &   2   &   3   &   4   &   5   \\ \hline
then $r\leqslant$                             &   54  &   11  &   13  &   15  &   20  &   13
\end{tabular}
\end{center}
\mycap{Upper bounds on $r$ depending on the congruence class of $d$ modulo $6$.\label{cong:bounds:tab}}
\end{table}
\end{lem}

\begin{proof}
The conclusion is obtained
by direct computation after imposing the appropriate $d_{\max}^{\,\equiv k\,(n)}(r)\geqslant 13$.
\end{proof}

Combining the restrictions given by Lemmas~\ref{floor:dmax:lem},~\ref{dcong:lem}
and~\ref{dcong:bounds:lem} we can now conclude our analysis of case (\textbf{A}).

\begin{prop}
No candidate orbifold cover $S(\alpha,\beta,\gamma,\delta)\dotsto^{d:1} S(2,3,r)$
exists with $d\geqslant 13$ and $r\geqslant 15$.
\end{prop}

\begin{proof}
Suppose first that $r\geqslant 19$. Then
Lemma~\ref{floor:dmax:lem} implies that $d\leqslant 15$, so $d$ can attain the values
$13$, $14$ and $15$, and Lemma~\ref{dcong:bounds:lem}
implies that $r\leqslant 15$, a contradiction.
For $15\leqslant r\leqslant 18$ Lemma~\ref{floor:dmax:lem} implies that $d\leqslant 17$,
and then Lemma~\ref{dcong:bounds:lem} shows that either $d=16$ or $d=r=15$.
For $d=16$ we then get a contradiction invoking Lemma~\ref{dcong:lem}
because $d_{\max}^{\,\equiv1\,(3)}(r)<16$ for $15\leqslant r\leqslant 18$.
Similarly for $d=r=15$ we get a contradiction because
$d_{\max}^{\,\equiv1\,(2)}(15)=13<15$.
\end{proof}

\begin{prop}
The only candidates $S(\alpha,\beta,\gamma,\delta)\dotsto^{d:1} S(2,3,r)$
with $d\geqslant 13$ and $13\leqslant r\leqslant 14$ are items \emph{\texttt{100}}
and \emph{\texttt{109}} in Table~\ref{hyp:S(4)-to-S(3):summary:83-116:tab}.
\end{prop}

\begin{proof}
Start with $r=14$. Lemma~\ref{floor:dmax:lem} implies that $d\leqslant 18$, whence
Lemma~\ref{dcong:bounds:lem} implies that $d$ can attain the values $15$, $16$ and $18$.
For $d=15$ we have $d\leqslant d_{\max}^{\,\equiv1\,(2)}(14)=13.5$ and for $d=16$ we have
$d\leqslant d_{\max}^{\,\equiv1\,(3)}(14)=15.25$, whence a contradiction in both cases.
For $d=18$ we note that the only partitions $\Pi$ of $d=18$ consisting of divisors of
$14$ and such that $c(\Pi)\leqslant 4$ are
\begin{center}
$(7,7,2,2),\quad (14,2,2),\quad (14,2,1,1),\quad (14,1,1,1,1)$.
\end{center}
For $(7,7,2,2)$ and $(14,1,1,1,1)$ we have $c(\Pi)=4$, so the other two
partitions of $d=18$ must be $(2,\ldots,2)$ and $(3,\ldots,3)$, but the
total length is $d+2=20$ only for $(14,1,1,1,1)$, which gives
\texttt{109}. For $(14,2,2)$ we have $c(\Pi)=2$, so the two other partitions
must be $(2,\ldots,2,1,1)$ and $(3,\ldots,3)$, but then the total length is $19$.
For $(14,2,1,1)$ we have $c(\Pi)=3$ and no choice is possible for
the two other partitions.

For $r=13$ again we have $d\leqslant 18$, and using as above the values of $d_{\max}^{\,\equiv k\,(n)}(13)$
we see that $d=14$, $d=15$ and $d=17$ are impossible. For $d=16$ the only partition $\Pi$ of $d$
consisting of divisors of $13$ with $c(\Pi)\leqslant 4$ is $(13,1,1,1)$, which implies
that the two other partitions must be $(2,\ldots,2)$ and $(3,\ldots,3,1)$, whence
\texttt{100}. For $d=18$ there is no partition $\Pi$ of $d$
consisting of divisors of $13$ with $c(\Pi)\leqslant 4$.
\end{proof}

The rest of the discussion leading to the proof of Proposition~\ref{hyp:S(4)-to-S(3):d_geq_13:prop} is
now quite similar to the arguments already used. For a decreasing value of $r$ between $12$ and $7$
one has an increasingly complicated argument consisting of:
\begin{itemize}
\item A use of Lemmas~\ref{floor:dmax:lem} to~\ref{dcong:bounds:lem}, to exclude some values of $d$;
\item The analysis of what partitions $\Pi_1,\Pi_2,\Pi_3$ of $2,3,r$ satisfy
${\rm l.c.m.}(\Pi_1)=2$, ${\rm l.c.m.}(\Pi_2)=3$, ${\rm l.c.m.}(\Pi_3)=r$,
$c(\Pi_1)+c(\Pi_2)+c(\Pi_3)=4$ and $\ell(\Pi_1)+\ell(\Pi_2)+\ell(\Pi_3)=d+2$; this last discussion is
easier for $r=7$ and $r=11$, since $\Pi_3$ can only consist of $r$'s and $1$'s.
\end{itemize}

We address the reader to~\cite{hyp1:proofs} for a careful description of this argument, only
mentioning that for $7\leqslant r\leqslant 12$ one gets exactly the 53 items in
Tables~\ref{hyp:S(4)-to-S(3):summary:83-116:tab}
and~\ref{hyp:S(4)-to-S(3):summary:117-146:tab} excluding the $9$
coming from case (\textbf{B}) and the 2 already found in case (\textbf{A}).
More precisely for $r=7,\ 8,\ 9,\ 10,\ 11,\ 12$ one gets respectively
$13,\ 22,\ 5,\ 7,\ 1,\ 5$ candidates, for a total of $53$.
\end{proof}

Combining Propositions~\ref{hyp:S(4)-to-S(3):d_leq_12:prop} and~\ref{hyp:S(4)-to-S(3):d_geq_13:prop}
we obtain the conclusion of the proof of
Theorem~\ref{hyp:S(4)-to-S(3):enum:thm}.
\finedimo

\dimo{hyp:T(1)-to-S(3):enum:thm}
We use the same notation as above.
Since $\Sigmatil$ is now the torus $T$ and $\Sigma$ is the sphere $S$,
the Riemann-Hurwitz formula~(\ref{RH:general:eq}) reads
\begin{equation}
\ell(\Pi_1)+\ell(\Pi_2)+\ell(\Pi_3)=d.\label{RH:T-to-S(3):eq}
\end{equation}
Therefore, we need to enumerate the degrees $d$, and the partitions $\Pi_1,\Pi_2,\Pi_3$ satisfying~(\ref{RH:T-to-S(3):eq}) and
\begin{equation}
c(\Pi_1)+c(\Pi_2)+c(\Pi_3)=1,\label{c=1:eq}
\end{equation}
because in this case in the associated candidate orbifold cover
$\Xtil\dotsto X$
we automatically have that $\Xtil=T(\alpha)$ with $\alpha>1$ is hyperbolic, so is $X$.

Relations~(\ref{RH:T-to-S(3):eq}) and~(\ref{c=1:eq}) imply that
for any given $d$ we must find divisors $p>1$ and $q>1$ of $d$,
an integer $r>1$ and a divisor $r'\neq r$ of $r$ such that $d-r'$ is a
multiple of $r$ and $\frac dp+\frac dq+\frac{d-r'}r+1=d$.
This is very easily done for $d\leqslant 17$ and leads to the first
$18$ items in Table~\ref{hyp:T(1)-to-S(3):summary:tab}.
(We note in passing that in Table~\ref{hyp:T(1)-to-S(3):summary:tab}
the partitions $\Pi_1,\Pi_2,\Pi_3$ defining a candidate cover are
rearranged for increasing l.c.m.) As an only
example, we present the argument for $d=12$, addressing the reader to~\cite{hyp1:proofs}
for the other degrees up to $17$.

So, for $d=12$, we note that the pairs $(r,r')$ with $r'\neq r$
a divisor of $r$ and $12-r'$ a multiple of $r$ are the following ones:
\begin{center}$(8,4),\quad (9,3),\quad (10,2),\quad (11,1)$.\end{center}
For each of them we have $\frac{12-r'}r=1$, so we must now
find two divisors $p$ and $q$ of $12$ such that $\frac{12}p+\frac{12}q=10$,
and one readily sees that up to permutation the only choice is $p=2$ and $q=3$.
We conclude that for $d=12$ there are $4$ relevant candidate covers, listed as
items \texttt{159} to \texttt{162} in Table~\ref{hyp:T(1)-to-S(3):summary:tab}.

Turning to the case $d\geqslant 18$, consider a candidate orbifold cover
$\Xtil\dotsto^{d:1} X$ with
hyperbolic $\Xtil=T(\alpha)$ and $X=S(p,q,r)$. Since
$0<-\chiorb(\Xtil)=1-\frac1\alpha<1$
and $\chiorb(\Xtil)=d\myprod\chiorb (X)$, we readily deduce that
$\frac{17}{18}<\frac1p+\frac1q+\frac1r<1$, which implies that either $p=2$, $q=4$ and $r=5$ or
$p=2$, $q=3$ and $7\leqslant r\leqslant 8$.
For $p=2$, $q=4$ and $r=5$ one must have $\alpha\in\{2,4,5\}$ and correspondingly
$d\in\{10,15,16\}$, which contradicts $d\geq 18$.
For $p=2$, $q=3$ and $r=7$ the partitions
of $d$ giving rise to the candidate must have one of the following forms:
\begin{center}
$(2,\ldots,2,1),(3,\ldots,3),(7,\ldots,7),$\\
$(2,\ldots,2),(3,\ldots,3,1),(7,\ldots,7),$\\
$(2,\ldots,2),(3,\ldots,3),(7,\ldots,7,1).$
\end{center}
Correspondingly,~(\ref{RH:T-to-S(3):eq}) translates into
\begin{center}
$\frac{d-1}2+1+\frac d3+\frac d7=d \quad\Rightarrow \quad d=21$\\
$\frac d2+\frac {d-1}3+1+\frac d7=d \quad\Rightarrow \quad d=28$\\
$\frac d2+\frac d3+\frac {d-1}7+1=d \quad\Rightarrow \quad d=36$
\end{center}
and we get the candidates \texttt{166} to \texttt{168} in Table~\ref{hyp:T(1)-to-S(3):summary:tab}.
For $r=8$ one carries out a similar analysis, this time with five different triples of partitions
(because $8$ has three divisors smaller than itself, while $2$ and $3$ have one), and one
finds as the only new candidate item \texttt{165} in Table~\ref{hyp:T(1)-to-S(3):summary:tab}.
Note that for the partitions
$(2,\ldots,2),(3,\ldots,3),(8,\ldots,8,1)$, imposing $\frac d2+\frac d3+\frac{d-1}8+1=d$, one finds
$d=21$, but this does not give a candidate, since $\frac d2$ and $\frac{d-1}8$ are not integers.
See the details in~\cite{hyp1:proofs}.
\finedimo

\section{Overview of the techniques used to prove\\
realizability and exceptionality}
\label{techniques:sec}

In this section we briefly present the methods using which we have proved realizability
or exceptionality for each of the 168 candidate covers in Theorems~\ref{hyp:S(4)-to-S(3):summary:thm}
and~\ref{hyp:T(1)-to-S(3):summary:thm}.

\paragraph{Dessins d'enfant (\textsf{DE})}
This is a classical technique, introduced by Gro\-thendieck in~\cite{Groth}
for studying algebraic maps between Riemann surfaces, which
proves a powerful tool both to exhibit realizability
and to show exceptionality of candidate covers.
We introduce Grothendieck's dessins in their original form, that is
valid only for covers of the sphere with three
branching points, but we mention that the method was generalized in~\cite{PePe1} to
the case of more branching points.

To begin recall that a bipartite graph is a
finite 1-complex whose vertex set is split
as $V_1\sqcup V_2$ and each edge has one endpoint
in $V_1$ and one in $V_2$. We now give the following:

\begin{defn} A \emph{dessins d'enfant} on a surface $\Sigmatil$
is a \emph{bipartite} graph $D\subset\Sigmatil$ such that
$\Sigmatil \backslash D$ consists of open discs. The \emph{length}
of one of these discs is the number of edges of $D$ along
which its boundary passes, counted with multiplicity.
\end{defn}

The connection between dessins d'enfant and branched covers comes from the
next result (see~\cite{PePe1} for a proof):

\begin{prop}
The realizations of a candidate
$\Sigmatil\argdotstosex{d:1}{(d_{11},\ldots,d_{1m_1}),
\ldots,(d_{31},\ldots,d_{3m_3})}S$
correspond to the dessins d'enfant $D\subset\Sigmatil$ with
set of vertices $V_1\sqcup V_2$ such that for $i=1,2$
the vertices in $V_i$ have valences $(d_{ij})_{j=1}^{m_i}$,
and the discs in $\Sigmatil \backslash D$ have lengths
$\left(2d_{3j}\right)_{j=1}^{m_3}$.
\end{prop}

\paragraph{Graph moves (\textsf{GM})}
This method will be used below only to prove exceptionality.
Its description here is rather generic, and in a sense obvious,
but in several practical cases we can indeed make the method work.

\begin{prop}\label{GM:prop}
Let $c$ and $c'$ be partial dessins d'enfant. Let $S\argdotsto{d:1}{\Pi}S$ be a
candidate surface branched cover. Suppose that:
\begin{itemize}
\item[(1)] While trying to construct a dessin d'enfant realizing
$S\argdotsto{d:1}{\Pi}S$ one is forced to insert a portion $c$;
\item[(2)] Any completion of $c$ to a dessin d'enfant realizing $S\argdotsto{d:1}{\Pi}S$, if any,
could also be used to complete $c'$ and would give a to a dessin d'enfant realizing
another candidate $S\argdotsto{d':1}{\Pi'}S$.
\end{itemize}
If $S\argdotsto{d':1}{\Pi'}S$ is exceptional then $S\argdotsto{d:1}{\Pi}S$ also is.
\end{prop}

The key point of this result is that one can establish condition (2) by examining $c$, $c'$ and $S\argdotsto{d:1}{\Pi}S$
only, without searching the completions of $c$ realizing $S\argdotsto{d:1}{\Pi}S$. We also
mention here the remarkable fact that in Proposition~\ref{28:41:by:GM:prop}
we will apply the \textsf{GM} method with $c'$ more complicated than $c$,
namely with $d'$ greater than $d$.

\paragraph{Very even data (\textsf{VED}) and block decompositions (\textsf{BD})}
If $d=d'+d''$, we will say that a partition $\Pi$ of $d$ \emph{refines} the partition
$(d',d'')$ if $\Pi$ splits as $\Pi'\sqcup\Pi''$ with $\Pi'$ a partition of $d'$
and $\Pi''$ a partition of $d''$. The next result established in~\cite{PePe1}
will be used to show exceptionality of several candidates:

\begin{prop}\label{VED:prop}
Consider a candidate surface cover $\Sigmatil\argdotstoter{d:1}{\Pi_1,\Pi_2,\Pi_3}\Sigma$
with $d$ and each element of $\Pi_{i}$ for $i=1,2$ being even.
If the candidate is realizable then $\Pi_3$ must refine the partition $\left(d/2,d/2\right)$.
\end{prop}

The next result was also shown in~\cite{PePe1} and it is based on the same
idea that under certain divisibility assumptions a surface branched cover
can be expressed as the composition of two other ones:

\begin{prop}\label{thm:1.4:prop}
Consider a candidate surface cover $\Sigmatil\argdotstoter{d:1}{\Pi_1,\Pi_2,\Pi_3}\Sigma$
with $d$ and each entry of $\Pi_{i}$ for $i=1,2$ being divisible by some $k$.
If the candidate is realizable then each entry of $\Pi_3$ is less than or equal to $d/k$.
\end{prop}

We mention that some more specific realizability criteria were provided
in~\cite{PePe1} for even $d$ and $\Pi_1=(2,\ldots,2)$ and either $\Pi_2=(5,3,2,\ldots,2,1)$, or $\Pi_2=(3,3,2,\ldots,2)$
or $\Pi_2=(3,2,\ldots,2,1)$. In an earlier version of~\cite{PePe1} the technique
leading to Propositions~\ref{VED:prop} and~\ref{thm:1.4:prop} was also
generalized to a certain theory of \emph{block decompositions},
which allows in particular to prove exceptionality of several candidate surface covers in degree $12$.
Since these candidates can be alternatively and quite easily discussed using \textsf{DE}, we will
refrain from restating these specific results here.

\paragraph{Geometric gluings (\textsf{GG})}
The main idea of~\cite{PaPe} was to use the geometry of $2$-orbifolds to analyze candidate
surface covers. But this was actually done only in the spherical and Euclidean
case, while for the $11$ hyperbolic candidates corresponding to
orbifold covers between triangular orbifolds the technique of \textsf{DE} was only used.
In this paper for the first time we actually apply the geometry of hyperbolic
orbifolds to discuss realizability of candidate surface branched covers.
The statement we give here, just like Proposition~\ref{GM:prop},
is a rather straight-forward one, but we can
actually employ it in several concrete examples, to
show both realizability and exceptionality:

\begin{prop}\label{GG:prop}
Consider a candidate surface branched cover $\Sigmatil\argdotstobis{d:1}{\Pi_1,\Pi_2,\Pi_3}S$
with associated candidate orbifold cover $\Xtil\dotsto S(p,q,r)$ of hyperbolic type.
Let $D$ be a fundamental domain for $S(p,q,r)$ obtained by mirroring
the hyperbolic triangle $T(p,q,r)$ with inner angles $\frac{\pi}p,\frac{\pi}q,\frac{\pi}r$
in one of its edges. Then the candidate is realizable if an only if
one can realize $\Xtil$ by gluing $d$ copies of $D$
along orientation-reversing hyperbolic isometries, in such a way that, upon
mapping each copy of $D$ to the original $D$,
the resulting orbifold cover $\Xtil\to S(p,q,r)$ matches the covering instructions
given by the original cover.
\end{prop}

\paragraph{Monodromy representation (\textsf{MR})}
We recall here a very classical viewpoint dating back to~\cite{Hurwitz}
(see also~\cite{PePe1}), based on the remark that
a realization of a candidate branched cover
$\Sigmatil\argdotstosex{d:1}{(d_{11},\ldots,d_{1m_1}),
\ldots,(d_{n1},\ldots,d_{nm_n})}\Sigma$
corresponds to the choice of its monodromy, which is a (suitable)
representation of the fundamental group of the $n$-punctured sphere
into the symmetric group $\permu_d$. More precisely, one has that
that a realization of the given candidate cover
corresponds to the choice of permutations
$\sigma_1,\ldots,\sigma_n\in\permu_d$ such that:
\begin{itemize}
\item $\sigma_i$ has cycles of lengths $(d_{ij})_{j=1}^{m_i}$;
\item the product $\sigma_1\cdots\sigma_n$ is the identity;
\item $\left\langle\sigma_1,\ldots,\sigma_n\right\rangle<\permu_d$
acts transitively on $\{1,\ldots,d\}$.
\end{itemize}

For $n=3$, to apply this method in practice, one should fix a permutation
$\sigma_1$ with cycle lengths $(d_{1j})_{j=1}^{m_1}$, let $\sigma_2$ vary
in the conjugacy class of permutations of cycle lengths $(d_{2j})_{j=1}^{m_2}$
and check whether $\left\langle\sigma_1,\sigma_2\right\rangle$ is transitive
and $\sigma_1\cdot\sigma_2$ has cycle lengths $(d_{3j})_{j=3}^{m_1}$.
When the degree $d$ is high this can be computationally quite demanding,
but a C++ code~\cite{Monge} written by Maurizio Monge, building also on some formulae
proved in~\cite{Zheng}, refines this approach
for the special case of permutations of the form relevant to
Theorem~\ref{hyp:T(1)-to-S(3):summary:thm}. The program employs the
correspondence between Young diagrams and representations of the
symmetric group, automatically generating the $p$-core diagrams
(which are very easy for the particular permutations involved)
to find all the relevant characters. This, together
with some computational tricks again specific to the special permutations
of Table~\ref{hyp:T(1)-to-S(3):summary:tab}, allows the program to
establish within a handful of seconds the realizability of any candidate
of the appropriate type up to degree $200$.

\section{Realizability and exceptionality\\ of the relevant candidate covers}
\label{real:ex:sec}

In this section we discuss the realizability of all the $168$ candidates described in
Theorems~\ref{hyp:S(4)-to-S(3):enum:thm} and~\ref{hyp:T(1)-to-S(3):enum:thm},
thereby completing the proof of Theorems~\ref{hyp:S(4)-to-S(3):summary:thm}
and~\ref{hyp:T(1)-to-S(3):summary:thm}. We begin with the following:

\begin{prop}\label{hyp:S(4)-to-S(3):real:prop}
The $117$ candidate surface covers described in
Tables~\ref{hyp:S(4)-to-S(3):summary:1-12:tab}
to~\ref{hyp:S(4)-to-S(3):summary:117-146:tab} and indicated there to be realizable
are indeed realizable.
\end{prop}

\begin{proof}
For each of the 117 candidates we have been able to draw a dessin d'enfant proving realizability.
To avoid showing all of them, we will present here only one for each degree up to
$12$, one for each candidate having a prime degree greater than $12$, and some for
the largest degrees in the tables. See~\cite{hyp1:proofs} for all other realizable candidates.
In all our dessins, with notation as in Tables~\ref{hyp:S(4)-to-S(3):summary:1-12:tab}
to~\ref{hyp:S(4)-to-S(3):summary:117-146:tab}, we will associate white vertices
to the entries of partition $\Pi_1$ and black vertices to those in $\Pi_3$,
so the regions of the complement of the dessin will correspond to the entries of $\Pi_2$.

For degree up to $12$ the examples we have chosen to show correspond to items
\texttt{3}, \texttt{6}, \texttt{15}, \texttt{24},
\texttt{33}, \texttt{45}, \texttt{60}, and \texttt{62}.
The dessins proving that they are all realizable are provided in Figure~\ref{dessins:up:to:12:fig}.
\begin{figure}
    \begin{center}
    \includegraphics[scale=.6]{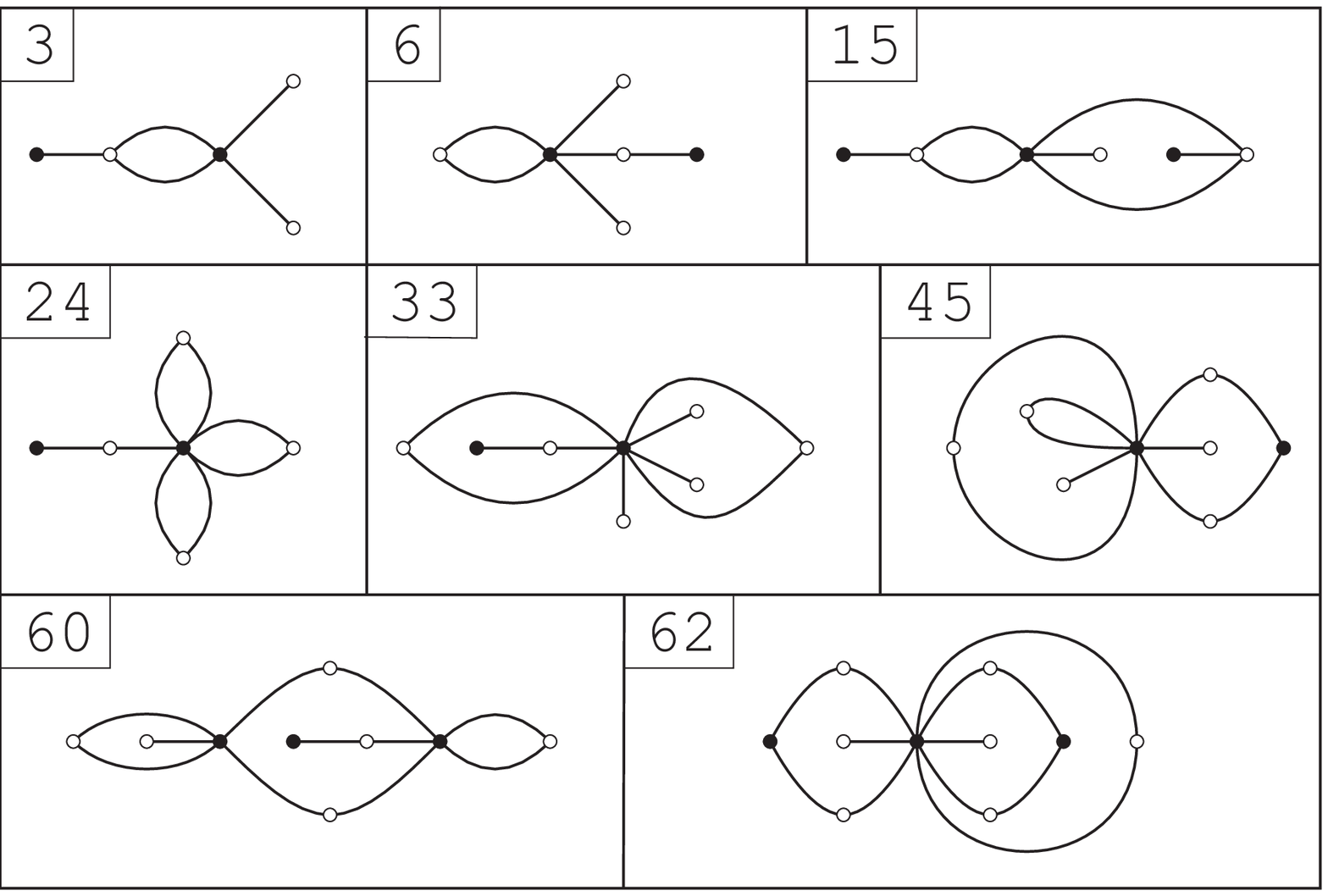}
    \end{center}
\mycap{Realization via \textsf{DE} of a sample of candidate covers for degree up to $12$.\label{dessins:up:to:12:fig}}
\end{figure}

For degree greater than $12$ the candidate covers with prime degree $d$ in
Tables~\ref{hyp:S(4)-to-S(3):summary:1-12:tab}
to~\ref{hyp:S(4)-to-S(3):summary:117-146:tab} are items
\texttt{83}, \texttt{84}, \texttt{85}, \texttt{86}, \texttt{105},
\texttt{106}, \texttt{117}, \texttt{137}, and \texttt{142}.
The dessins proving that they are all realizable are provided in Figure~\ref{dessins:more:than:12:fig}.
\begin{figure}
    \begin{center}
    \includegraphics[scale=.8]{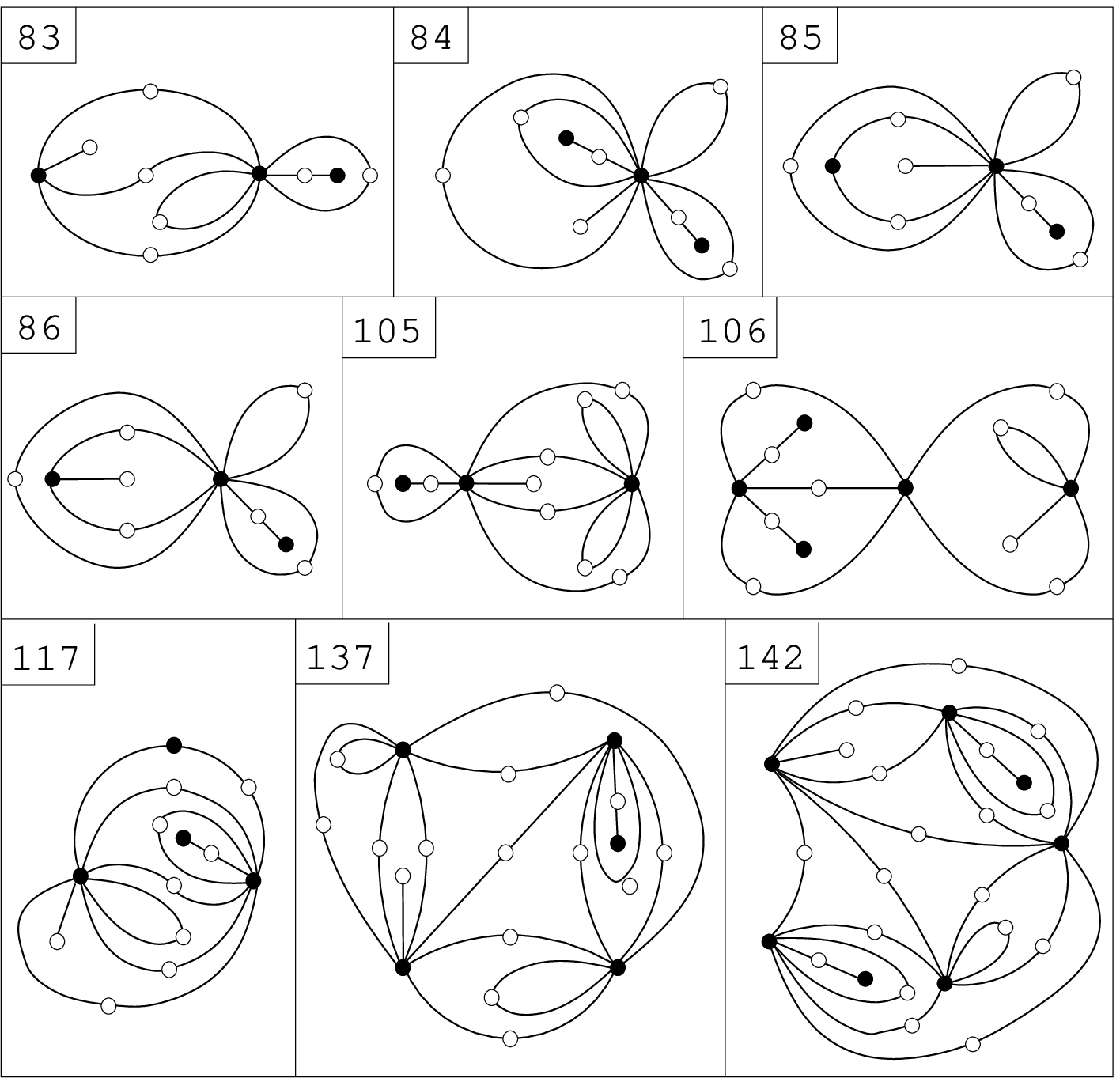}
    \end{center}
\mycap{Realization via \textsf{DE} of the candidate covers with prime degree larger than $12$.\label{dessins:more:than:12:fig}}
\end{figure}

For each of the candidates \texttt{143}, \texttt{144} and \texttt{145} a dessin d'enfant showing its realizability is shown
in Figure~\ref{dessins:deg:44:45:52:fig}.
\begin{figure}
    \begin{center}
    \includegraphics[scale=.64]{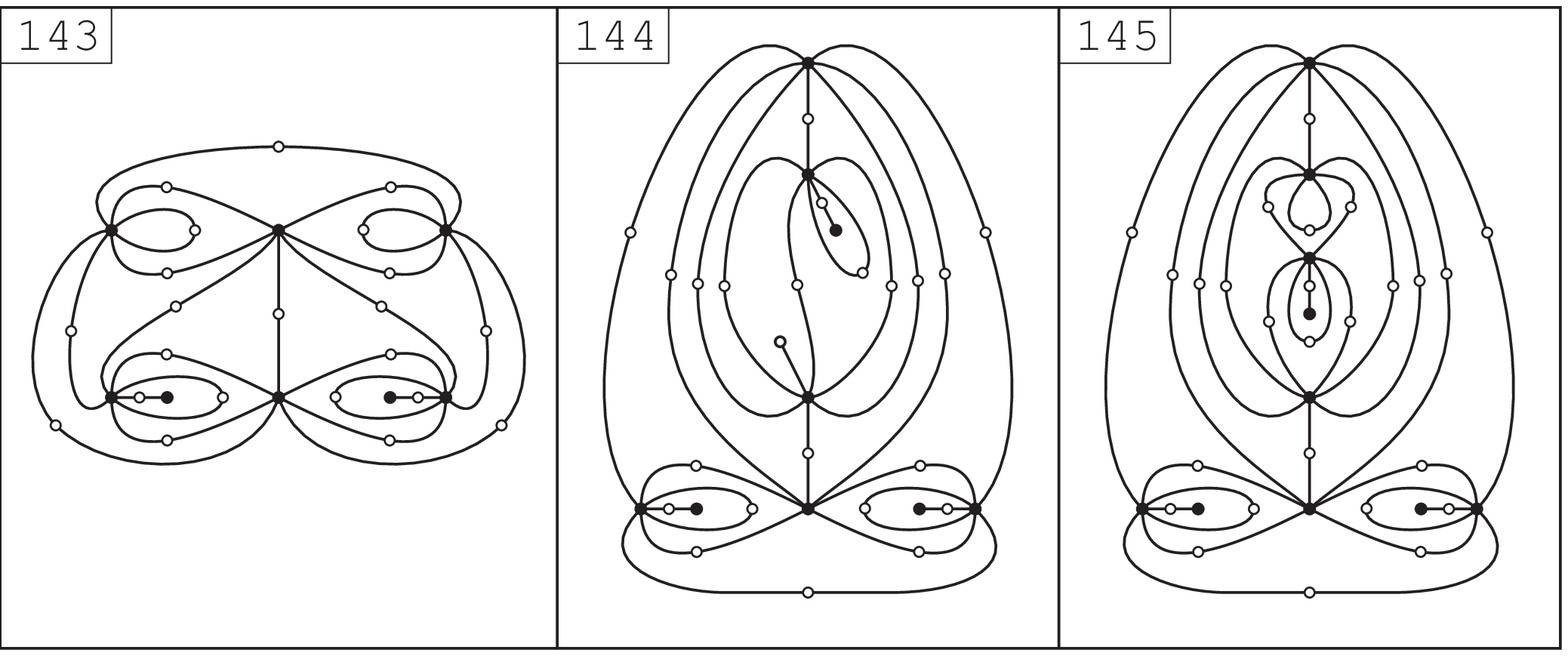}
    \end{center}
\mycap{Realization via \textsf{DE} of the candidate covers with degrees $44$, $45$ and $52$.\label{dessins:deg:44:45:52:fig}}
\end{figure}
Even if we have one, we do not show a dessin for candidate \texttt{146} because
a proof of its realizability will be given below in Proposition~\ref{GG:real:prop} using \textsf{GG}.\end{proof}

Let us turn to the following:

\begin{prop}\label{hyp:S(4)-to-S(3):excep:prop}
The $26$ candidate surface covers described in
Tables~\ref{hyp:S(4)-to-S(3):summary:1-12:tab}
to~\ref{hyp:S(4)-to-S(3):summary:117-146:tab} and indicated there to be exceptional
are indeed exceptional.
\end{prop}

\begin{proof}
For degree up to $20$ we could actually merely refer to the computer-generated census of Zheng~\cite{Zheng},
but we prefer to give theoretical proofs. To begin, we note that the \textsf{VED} criterion
of Proposition~\ref{VED:prop} shows exceptionality of candidates
\texttt{25}, \texttt{31}, \texttt{49}, \texttt{50}, \texttt{69}, \texttt{97}, \texttt{98}, \texttt{116}
and \texttt{126}. Candidates \texttt{35} and \texttt{36} are exceptional due to
Proposition~\ref{thm:1.4:prop}, while \texttt{23} and \texttt{28} due
to Propositions~1.3 and~1.2 of~\cite{PePe1}, respectively.
The \textsf{BD} criterion of~\cite[Section 5]{PePe1}
alluded to after Proposition~\ref{thm:1.4:prop} implies that candidates \texttt{72}-\texttt{78}
are exceptional. For degrees up to 20 this leaves out only candidates
\texttt{41}, \texttt{57}, \texttt{79}, \texttt{113}, \texttt{114} and
\texttt{115} for which we prove exceptionality using \textsf{DE}
in Figure~\ref{41:57:79:113:114:115:fig}.
Here we always associate white vertices to the entries of $\Pi_1$ and black ones
to those in $\Pi_2$, and we use their valences and the lengths of some of the
complementary regions to construct forced portions of dessin d'enfant in which
one sees offending lengths of the complementary regions and/or one finds it
impossible to get a connected dessin with the prescribed lengths of the complementary regions.
\begin{figure}
\centering
    \includegraphics[scale=.5]{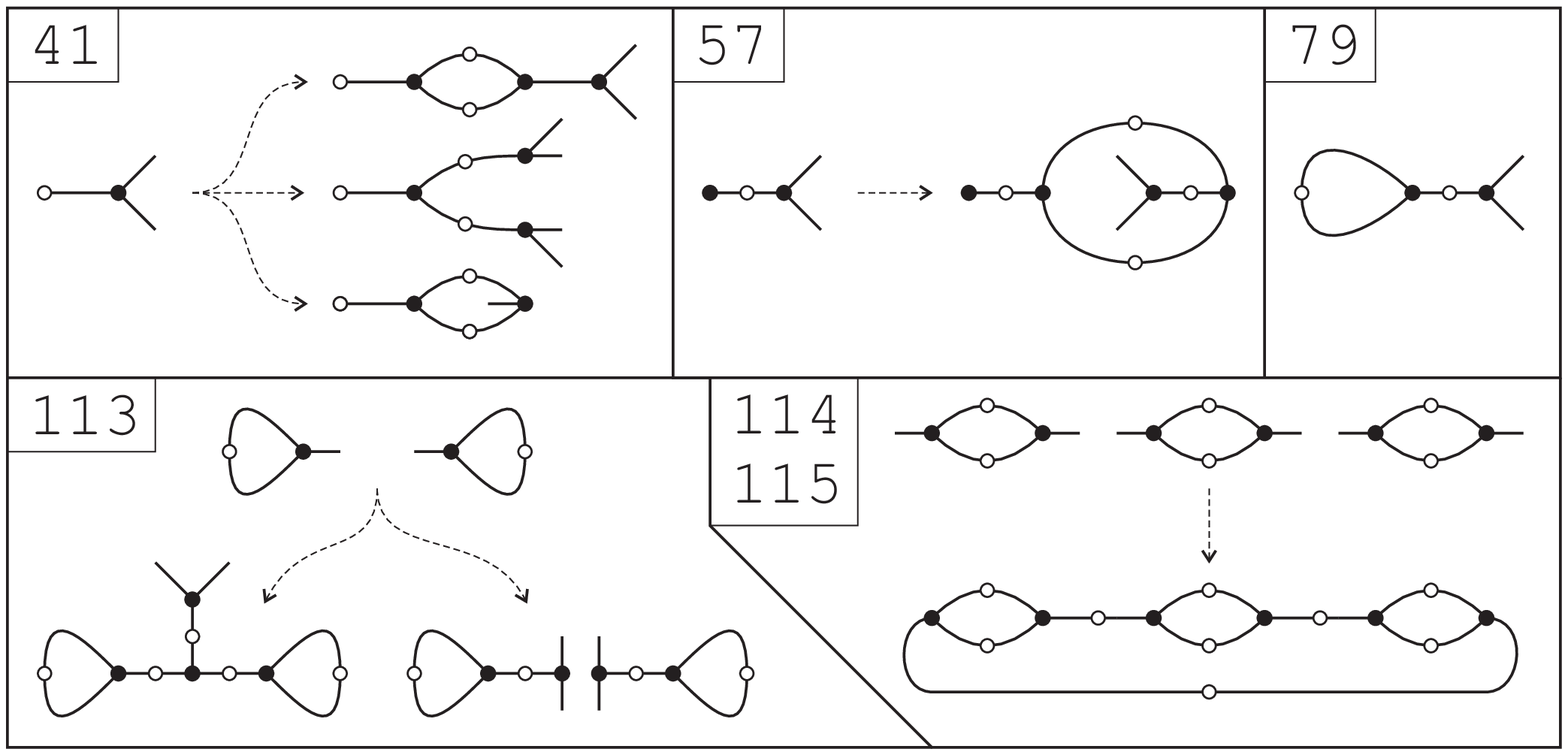}
\mycap{Exceptionality proofs via \textsf{DE}.\label{41:57:79:113:114:115:fig}}
\end{figure}

We are left to prove exceptionality of candidates \texttt{122}, \texttt{125} and \texttt{131},
in degrees $21$, $22$ and $24$ respectively. We start with \texttt{125}, for which we use
\textsf{DE} again. Let us try to construct a dessin with white vertices corresponding
to $\Pi_1=(2,\ldots,2)$ and black vertices corresponding to $\Pi_3=(8,8,4,1,1)$.
The proof that this is actually impossible is contained in Figure~\ref{125:fig}.
In part (\texttt{a}) we show that neither of the
black $1$-valent vertices can be joined to
the $4$-valent one (otherwise a region of length at least 4 would arise),
so both are joined to an $8$-valent one,
and in part (\texttt{b}) we show that the two $1$-valent black vertices
are joined to different $8$-valent black vertices
(for the same reason). Then in part (\texttt{c})
we show that for the region incident to only one black vertex, this cannot
be the $4$-valent one; note in particular that we use part (\texttt{b}) in the last passage.
This implies that there is an $8$-valent black vertex joined to a $1$-valent one
and incident to a region of length $2$, and in part (\texttt{d}) we show that this is impossible,
once again by contradicting the fact that all regions but one should have length $3$.
\begin{figure}
\centering
    \includegraphics[scale=.55]{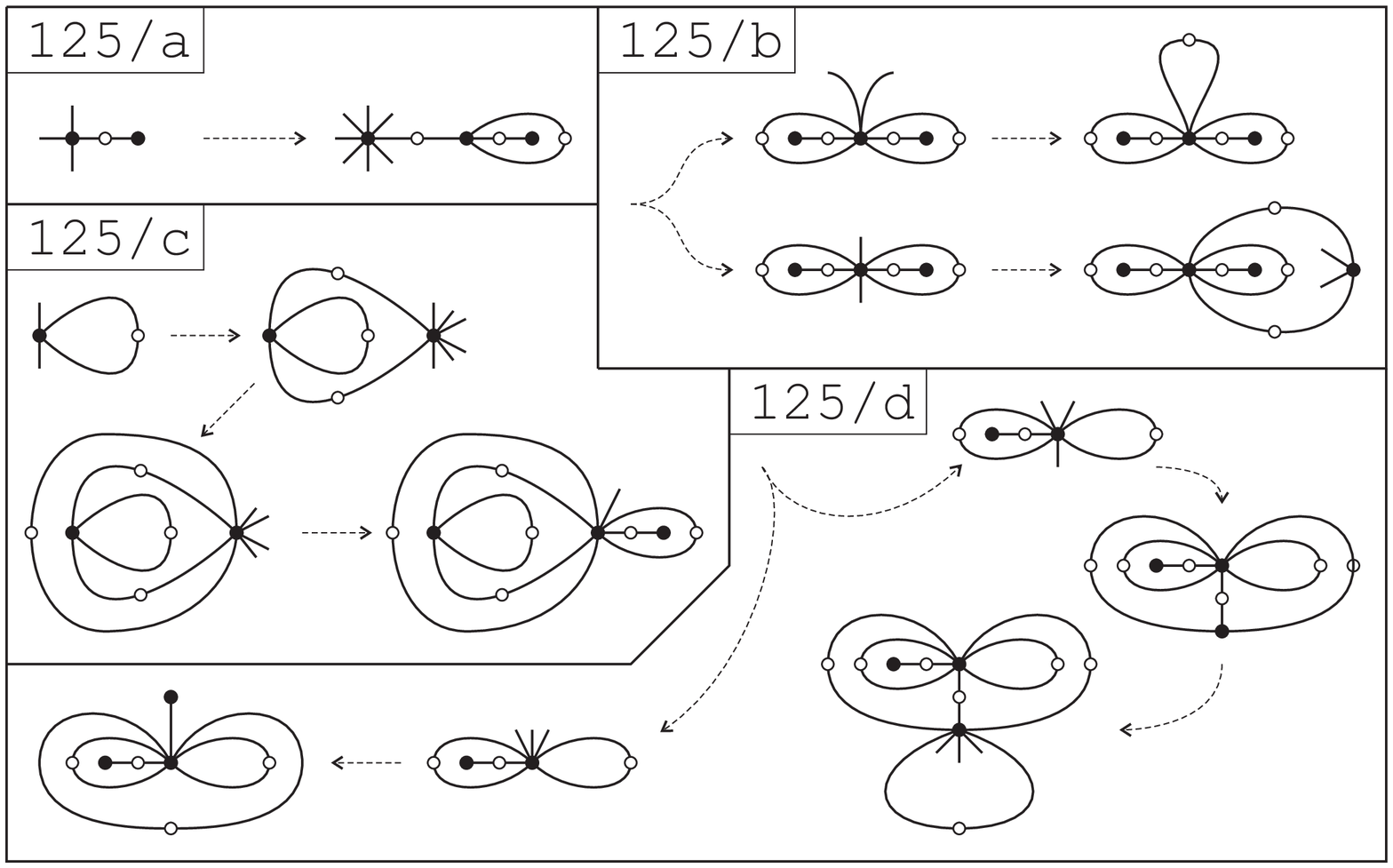}
\mycap{Exceptionality of \texttt{125} via \textsf{DE}.\label{125:fig}}
\end{figure}

To deal with \texttt{122} and \texttt{131} we will use the new techniques
\textsf{GM} and \textsf{GG} introduced in Section~\ref{techniques:sec},
but we prefer first to provide alternative proofs
of exceptionality and realizability in lower degree using these techniques,
to allow the reader to familiarize with them.
We begin with the following:

\begin{prop}\label{28:41:by:GM:prop}
The candidates \emph{\texttt{28}}, \emph{\texttt{41}} and \emph{\texttt{122}} can
be shown to be exceptional using the \emph{\textsf{GM}} technique.
\end{prop}

\begin{proof}
Recall that for \texttt{28}
the partitions are $\Pi_1=(2,2,2,2)$, $\Pi_2=(3,3,1,1)$ and $\Pi_3=(5,3)$, and let us try to construct
a dessin with white and black vertices associated to $\Pi_1$ and $\Pi_3$ respectively.
Since there are two regions of length $2$, at least
one of them has a $5$-valent black vertex, and we apply the move shown in
Figure~\ref{28:41:122:fig}. A dessin realizing \texttt{28} would then give
one realizing the candidate $S\argdotstoqua{10:1}{(2,\ldots,2),(3,3,3,1),(6,3,1)}S$,
which is exceptional by~\cite[Theorem 3.6]{PaPe}.
\begin{figure}
\centering
    \includegraphics[scale=.65]{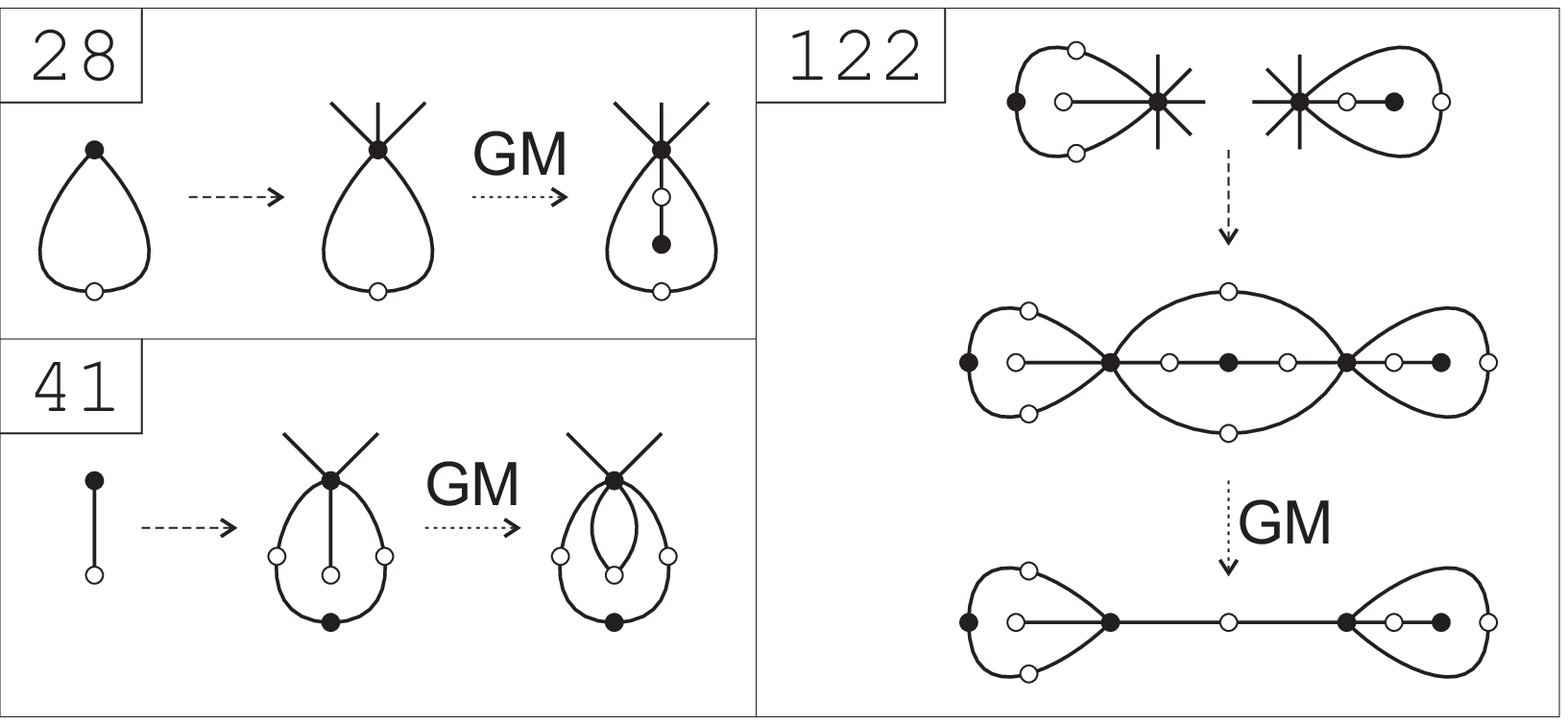}
\mycap{Exceptionality proofs via \textsf{GM}.\label{28:41:122:fig}}
\end{figure}

For \texttt{41} we proceed similarly, assigning white vertices to the partition $(2,2,2,2,1)$ and
black ones to $(5,2,2)$, leaving the partition $(3,3,3)$ for the regions. The $1$-valent
white vertex must be joined to a $5$-valent black one as in Figure~\ref{28:41:122:fig},
so we apply the move shown, which proves that if \texttt{41}
is realizable then $S\argdotstoqua{10:1}{(2,\ldots,2),(3,3,3,1),(6,2,2)}S$ also is,
which is false by the \textsf{VED} criterion.

Turning to \texttt{122} we recall that $d=21$,
$\Pi_1=(2,\ldots,2,1)$, $\Pi_2=(3,\ldots,3)$ and $\Pi_3=(8,8,2,2,1)$.
Again we use white for $\Pi_1$ and black for $\Pi_3$.
Two configurations as at the top of Figure~\ref{28:41:122:fig} must exist, and the only case in which the two edges
emanating from a valence-$2$ black vertex end on the same black vertex occurs in the first of these
configurations; therefore, a configuration as that to which we apply the move in
Figure~\ref{28:41:122:fig} occurs.
Note that at each valence-$8$ black vertex two emanating germs of edges are missing, because we apply
the move regardless of their position. After the move we get the candidate
$S\argdotstoqua{15:1}{(2,\ldots,2,1),(3,\ldots,3),(6,6,2,1)}S$, which is exceptional
by~\cite[Theorem 3.6]{PaPe}, and the proof is complete.\end{proof}

We now turn to the \textsf{GG} technique, that we first employ to provide
alternative proofs of three already shown realizability results, and of one
such result not explicitly given above. One more such realizability result
will be proved below in Proposition~\ref{GG:146:prop}.

\begin{prop}\label{GG:real:prop}
The candidates \emph{\texttt{1}}, \emph{\texttt{10}}, \emph{\texttt{13}}, and  \emph{\texttt{132}}
can be shown to be realizable using the \emph{\textsf{GG}} technique.
\end{prop}

\begin{proof}
For the first three candidates we show a realization in Figure~\ref{1:10:13:fig}, as we now explain.
\begin{figure}
\centering
    \includegraphics[scale=.75]{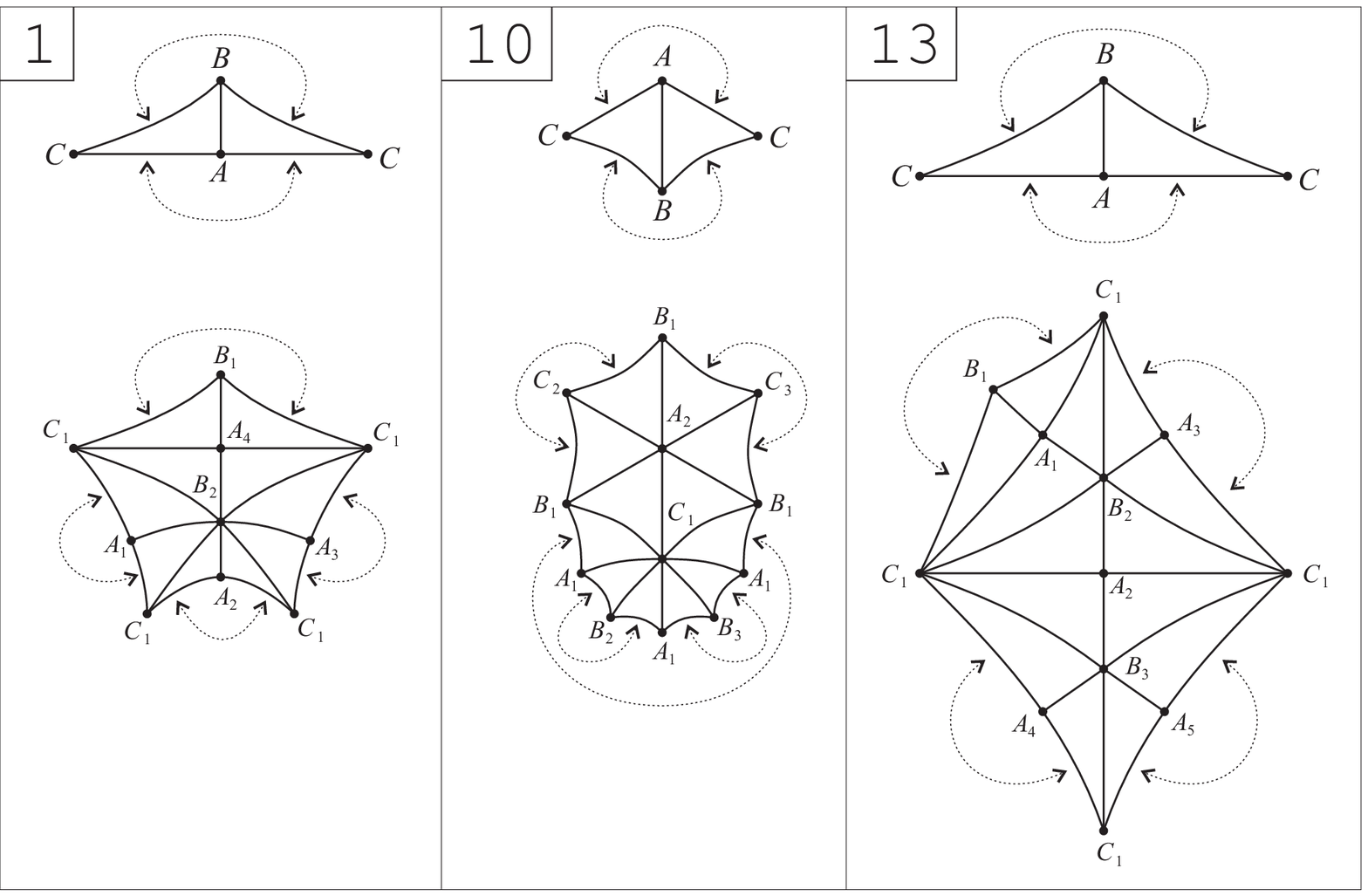}
\mycap{Realizability proofs via \textsf{GG}.\label{1:10:13:fig}}
\end{figure}
For \texttt{1} we should have a cover $S(2,2,2,4)\dotsto^{5:1} S(2,4,5)$
with covering instructions $(2,2,2)\dotsto 2$ and $4\dotsto 4$.
In the figure we show $S(2,4,5)$ as a gluing of two copies of the hyperbolic
triangle $T(2,4,5)$ with inner angles $\frac{\pi}2,\frac{\pi}4,\frac{\pi}5$.
The cone points of orders $2,4,5$ are respectively $A,B,C$.
And in the same figure we show $S(2,2,2,4)$ as a gluing of $10$ copies of $T(2,4,5)$.
This induces a covering $S(2,2,2,4)\myto^{5:1} S(2,4,5)$
with each $A_i,B_i,C_i$ mapped respectively to $A,B,C$. Since
the cone points of $S(2,2,2,4)$ are $A_1,A_2,A_3$ of order $2$ and $B_1$ of order $4$,
the required covering instructions are realized. Note that $A_4$ and $B_2$ are obviously
non-singular, and so is $C_1$, because there are $10$ angles $\frac{\pi}5$ incident to it.

The argument for \texttt{10} is similar. We must realize
$S(4,4,4,4)\dotsto^{6:1} S(3,4,4)$ with the instructions
$(4,4)\dotsto 4$ and $(4,4)\dotsto 4$ (which are automatic in this case).
In Figure~\ref{1:10:13:fig} the cone points of $S(3,4,4)$ are $A,B,C$ of orders $3,4,4$
respectively, and those of $S(4,4,4,4)$ are $B_2,B_3,C_2,C_3$ all of order $4$,
so the cover is as desired.

For \texttt{13} we must realize $S(2,2,2,3)\dotsto^{7:1} S(2,3,7)$
with, of course, $(2,2,2)\to 2$ and $3\to 3$. Here the cone points are $A,B,C$  of orders
$2,3,7$ in $S(2,3,7)$, and $A_3,A_4,A_5$ of order $2$ and $B_1$ of order $3$ in $S(2,2,2,3)$.

We now turn to \texttt{132}, which is treated in Figure~\ref{132:fig}.
\begin{figure}
\centering
    \includegraphics[scale=.75]{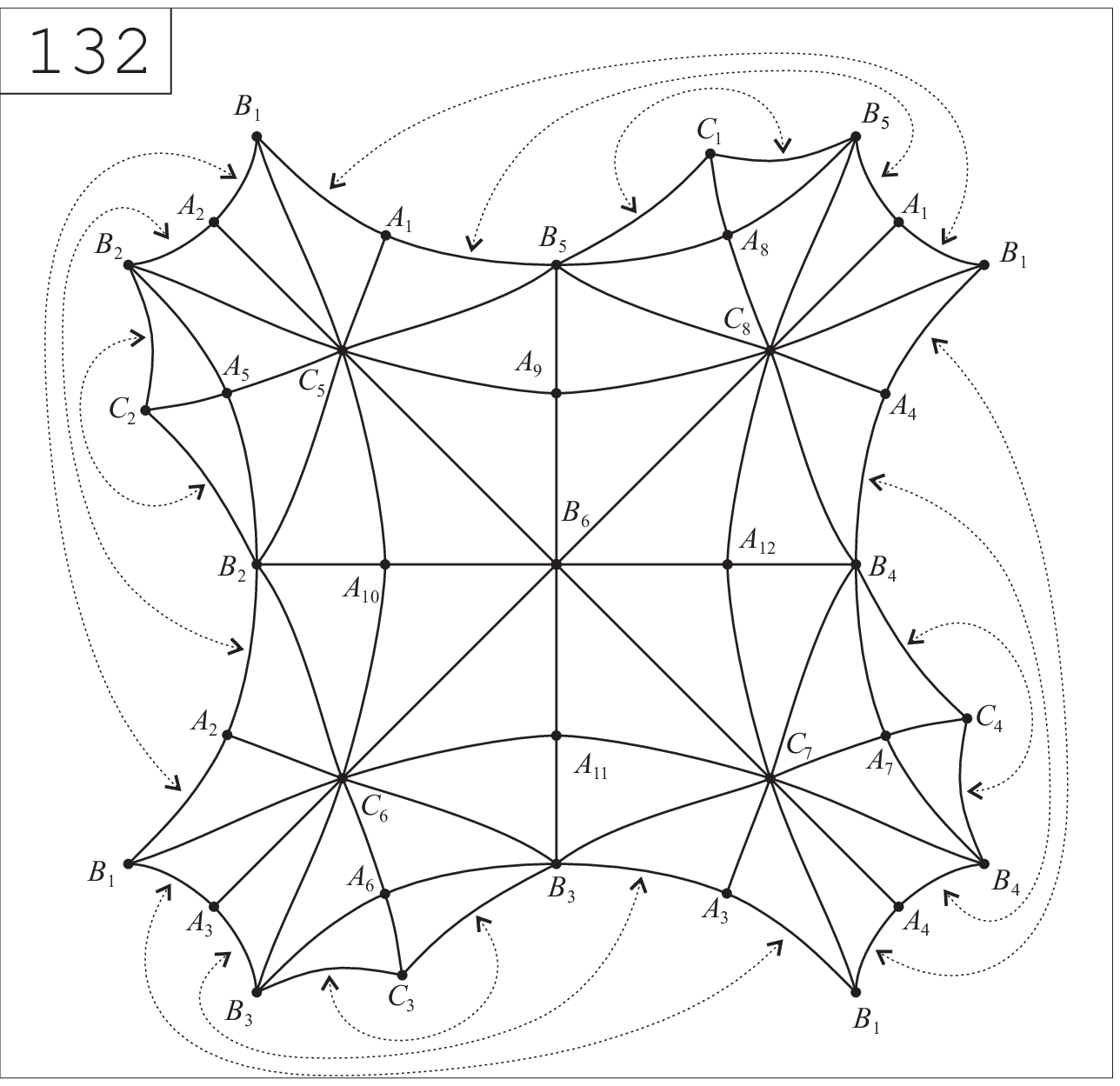}
\mycap{Realizability of candidate \texttt{132} via \textsf{GG}.\label{132:fig}}
\end{figure}
The cover to be realized is $S(5,5,5,5)\dotsto^{24:1}S(2,4,5)$ with (of course)
instructions $(5,5,5,5)\dotsto 5$. For $S(2,4,5)$ the fundamental domain is the same used for
\texttt{1}, and the reader can check that Figure~\ref{132:fig} contains $48$ copies
of $T(2,4,5)$ giving $S(5,5,5,5)$ with cone points at $C_1,C_2,C_3,C_4$.
\end{proof}

\begin{prop}\label{GG:excep:prop}
The candidates \emph{\texttt{23}} and \emph{\texttt{131}}
can be shown to be exceptional using the \emph{\textsf{GG}} technique.
\end{prop}

\begin{proof}
To \texttt{23} we associate $S(2,3,3,6)\dotsto^{8:1} S(2,4,6)$
with  $(2,3,3,6)\dotsto 6$. We show the fundamental domain $D$ of $S(2,4,6)$
in Figure~\ref{23:fig}, with cone points of order $2,4,6$ at $A,B,C$.
\begin{figure}
\centering
    \includegraphics[scale=.75]{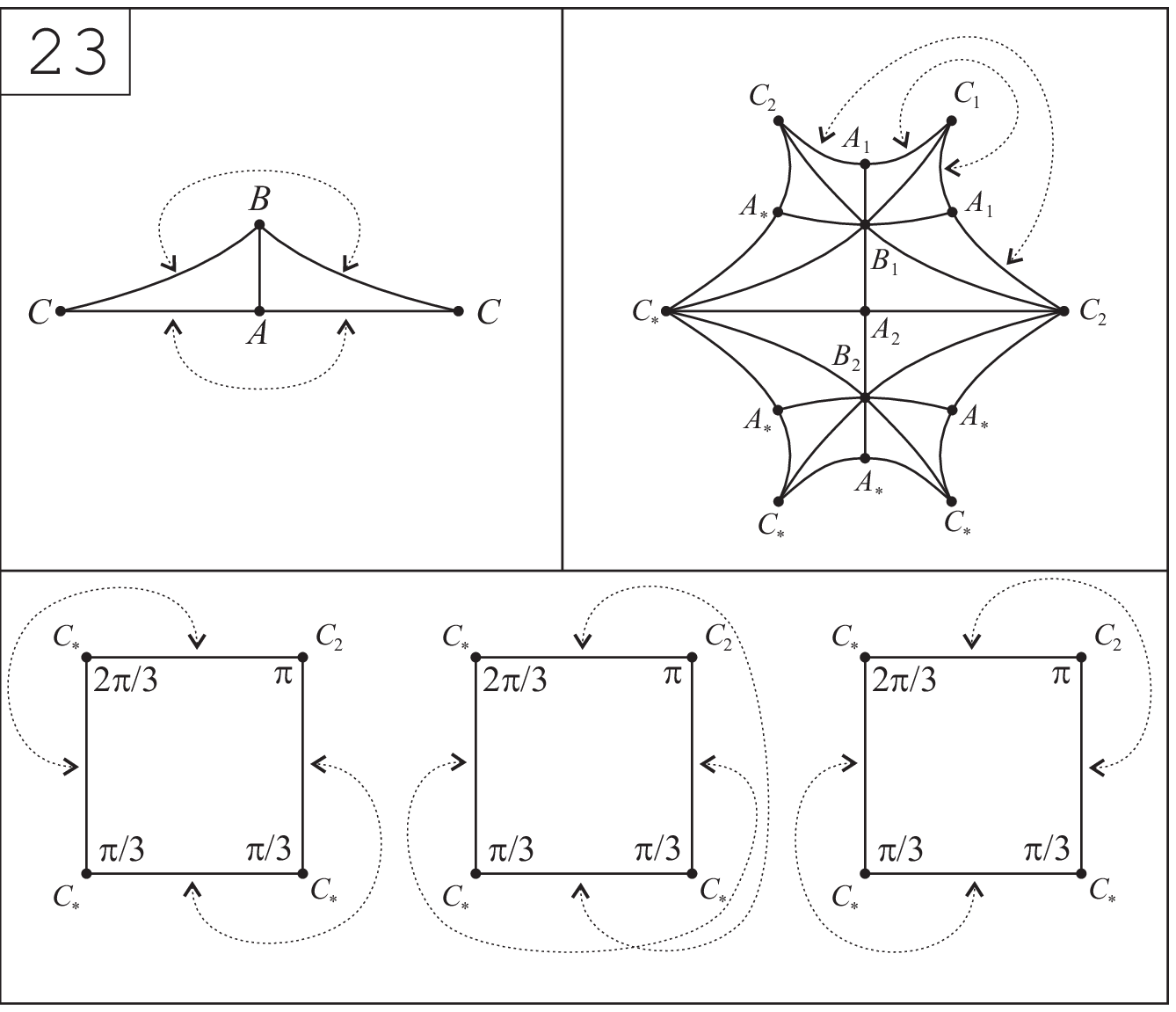}
\mycap{Exceptionality of candidate \texttt{23} via \textsf{GG}.\label{23:fig}}
\end{figure}
Since, in a realization of the candidate, $B$ is covered by two non-singular points
(and four singular ones),
of the $8$ copies of $D$ there must be $4$ around some $B_1$ and the other four around some $B_2$.
Now $A$ is also covered by non-singular points, which implies that no edge $A_*C_*$ is glued to the adjacent
$A_*C_*$. In other words, each (straight) segment $C_*A_*C_*$ can be regarded as a single
edge, and gets glued to another such segment. Since $S(2,3,3,6)$ is connected, the
two blocks of $4$ copies of $D$ already described are glued together, forming a
single block of $8$ copies of $D$ with $6$ free edges of type $C_*A_*C_*$.
But, to get a cone point of order $6$, at some top or bottom $C_*$,
say $C_1$, we must have a gluing
as shown in the figure. We now abandon hyperbolically correct pictures and use combinatorial ones,
but we keep track of geometry. To do this we note that the block of $8$ copies of $D$,
after performing the gluing shown, becomes a ``quadrangle'' (because we can ignore the $A_*$'s)
with inner angles $\pi,\frac23\pi,\frac\pi3,\frac\pi3$
(and an inner cone point of order $6$, not shown in the picture). As shown in Figure~\ref{23:fig},
there are now three ways to pair the edges of this quadrangle by orientation-reversing maps.
Since we have already realized a cone point at $C_1$ of angle $\frac{2\pi}6$, these gluings
give rise to a hyperbolic cone surface as follows:
\begin{itemize}
\item based on the sphere, with cone angles $\frac{2\pi}6,\frac{2\pi}6,\frac{2\pi}3,\frac{4\pi}3$;
\item based on the torus with cone angles $\frac{2\pi}6,\frac{7\pi}3$;
\item based on the sphere with cone angles $\frac{2\pi}6,\frac{2\pi}6,\frac{2\pi}2,\frac{2\pi}2$.
\end{itemize}
Neither of these is the desired $S(2,3,3,6)$, and the exceptionality of \texttt{23}
is proved. We note however that if one disregards geometry the previous gluings do give rise
of realizations of candidate branched covers of degree $8$, but not of the desired one. Namely:
\begin{itemize}
\item $S\argdotstoqua{8:1}{(2,2,2,2),(4,2,1,1),(4,4)}S$, with the associated Euclidean $S(2,4,4)\dotsto S(2,4,4)$;
\item $T\argdotstoqua{8:1}{(2,2,2,2),(7,1),(4,4)}S$, with associated $T(7)\dotsto S(2,4,7)$;
\item $S\argdotstoqua{8:1}{(2,2,2,2),(3,3,1,1),(4,4)}S$ with the associated spherical $S(3,3)\dotsto S(2,3,4)$.
\end{itemize}

Turning to $\texttt{131}$, the associated candidate is $S(3,9,9,9)\dotsto^{24:1} S(2,3,9)$ with instructions
$(3,9,9,9)\dotsto 9$. We show in Figure~\ref{131:fig} its fundamental domain $D$, with cone orders $2,3,9$ at $A,B,C$.
\begin{figure}
\centering
    \includegraphics[scale=.75]{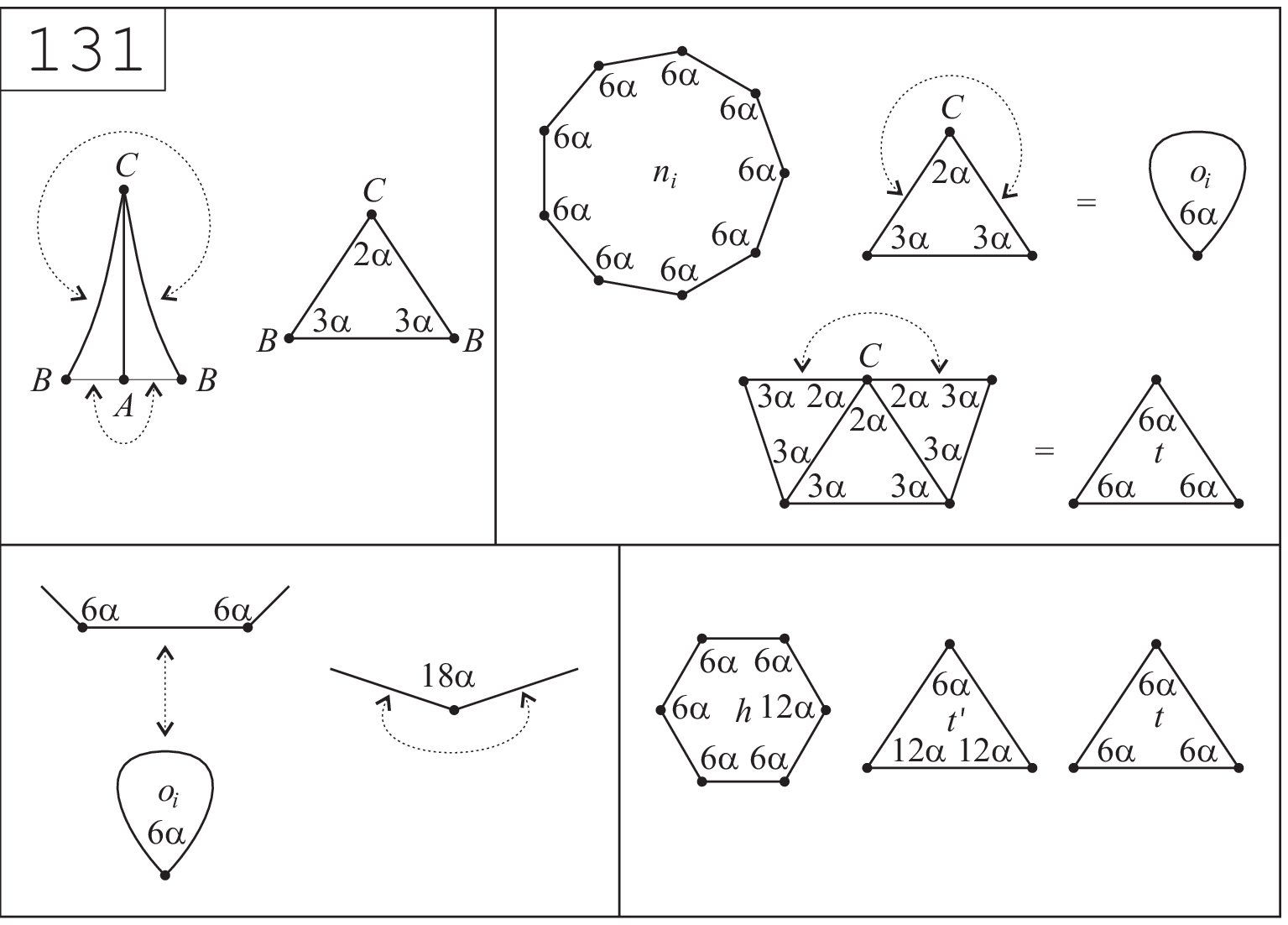}
\mycap{Exceptionality of \texttt{131} via \textsf{GG}.\label{131:fig}}
\end{figure}
As for \texttt{23}, using the fact that $A$ must be covered by non-singular points, we see we can forget $A$, and
redraw $D$ as a combinatorial triangle, but keeping track of its geometry by writing the angles at the
vertices, with $\alpha=\frac{\pi}9$. Since $C$ is covered by two non-singular points,
one of order $3$ and three of order $9$, the $24$ copies of $D$ can be grouped in two groups of $9$
copies giving the blocks $n_1,n_2$ with nine edges, 3 copies giving blocks $o_1,o_2,o_3$ with one
edge, and 3 copies giving a block $t$ with three edges. Note that at this stage all
vertices should cover $B$ and they all have angle $6\alpha=\frac23\pi$.
We should now assemble the blocks so that at each glued vertex the total angle is $18\alpha=2\pi$, because
all points covering $B$ are non-singular. This implies that gluing the edge of some
$o_i$ to another edge $e$ of some block forces the two edges of the block adjacent to
$e$ to be glued together. Therefore, no $o_i$ is glued to $t$, and (up to change of notation)
$o_1$ and $o_2$ are glued to $n_1$, while $o_3$ is glued to $n_2$. This gives two new blocks
$h$ and $t'$, with angles as shown, that we must now use with $t$.
The edge of $t'$ whose ends have angle $12\alpha$ must be glued to an edge whose ends have angle $6\alpha$,
and two more gluings are then forced. This implies that the said edge of $t'$ is not glued to $t$, and there
are two ways up to symmetry to glue it to an edge of $h$. One of them gives a triangle
with angles $6\alpha,6\alpha,24\alpha$, and the other one a triangle with angles $6\alpha,12\alpha,18\alpha$,
which easily implies that the process cannot be carried to the end.
\end{proof}

This concludes the proof of Proposition~\ref{hyp:S(4)-to-S(3):excep:prop}.
\end{proof}

As already announced we now show via \textsf{GG} that our highest degree candidate
is realizable, which was not explicitly done within the proof of Proposition~\ref{hyp:S(4)-to-S(3):real:prop}.

\begin{prop}\label{GG:146:prop}
The candidate \emph{\texttt{146}} can be shown to be realizable using the \emph{\textsf{GG}} technique.
\end{prop}

\begin{proof}
We must realize $S(7,7,7,7)\dotsto^{60:1} S(2,3,7)$ with $(7,7,7,7)\to 7$.
If $A,B,C$ are the cone points of orders $2,3,7$, the fact that $A$ is covered by non-singular
points implies as above that we can ignore it, then we must use $60$ copies of a triangle with vertices
$B,B,C$ and angles $\frac\pi3,\frac\pi3,\frac27\pi$. Taking into account the way $C$ is covered we then
get blocks $h_i$ for $i=1,\ldots,8$ and $o_i$ for $i=1,2,3,4$, as shown in Figure~\ref{146:fig},
with $\alpha=\frac\pi3$, that we must assemble creating non-singular points only,
that is, having angle $6\alpha=2\pi$.
\begin{figure}
\centering
    \includegraphics[scale=.75]{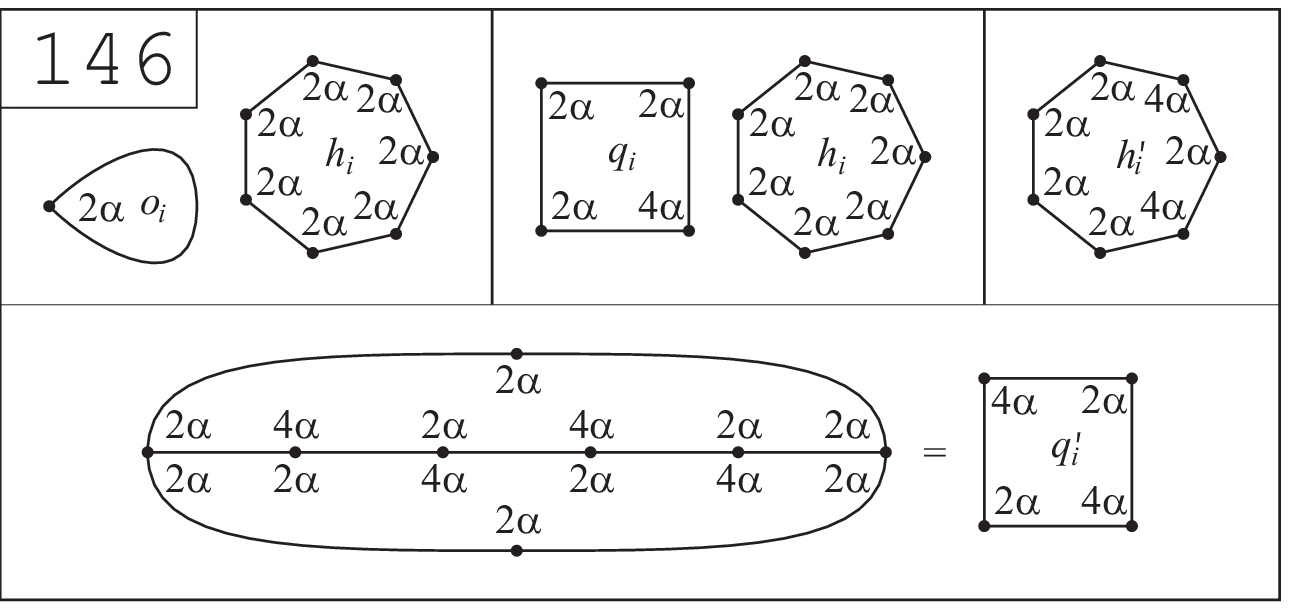}
\mycap{Realizability of \texttt{146} via \textsf{GG}.\label{146:fig}}
\end{figure}
Gluing $o_i$ to $h_{i+4}$
we get the blocks $q_i$ for $i=1,2,3,4$, and we still have $h_i$ for $i=1,2,3,4$.
Now we glue the two edges of $q_i$ separated by the vertex with angle $4\alpha$ to an edge of
$h_i$, getting the blocks $h'_i$ for $i=1,2,3,4$. We glue them in pairs as illustrated, getting
two blocks $q'_i$ for $i=1,2$, that we can now glue together to finish the process.
\end{proof}

The analysis of candidates with associated $S(\alpha,\beta,\gamma,\delta)\dotsto S(p,q,r)$ is complete, so
we turn to the case $T(\alpha)\dotsto S(p,q,r)$. We begin with:

\begin{prop}\label{hyp:T(1)-to-S(3):real:prop}
The $17$ candidate surface covers described in
Table~\ref{hyp:T(1)-to-S(3):summary:tab} and indicated there to be realizable
are indeed realizable.
\end{prop}

\begin{proof}
Again we simply exhibit one dessin d'enfant for each relevant candidate, which is done in
Figures~\ref{torus:real:fig:1} and~\ref{torus:real:fig:2}.
\begin{figure}
\centering
    \includegraphics[scale=.4]{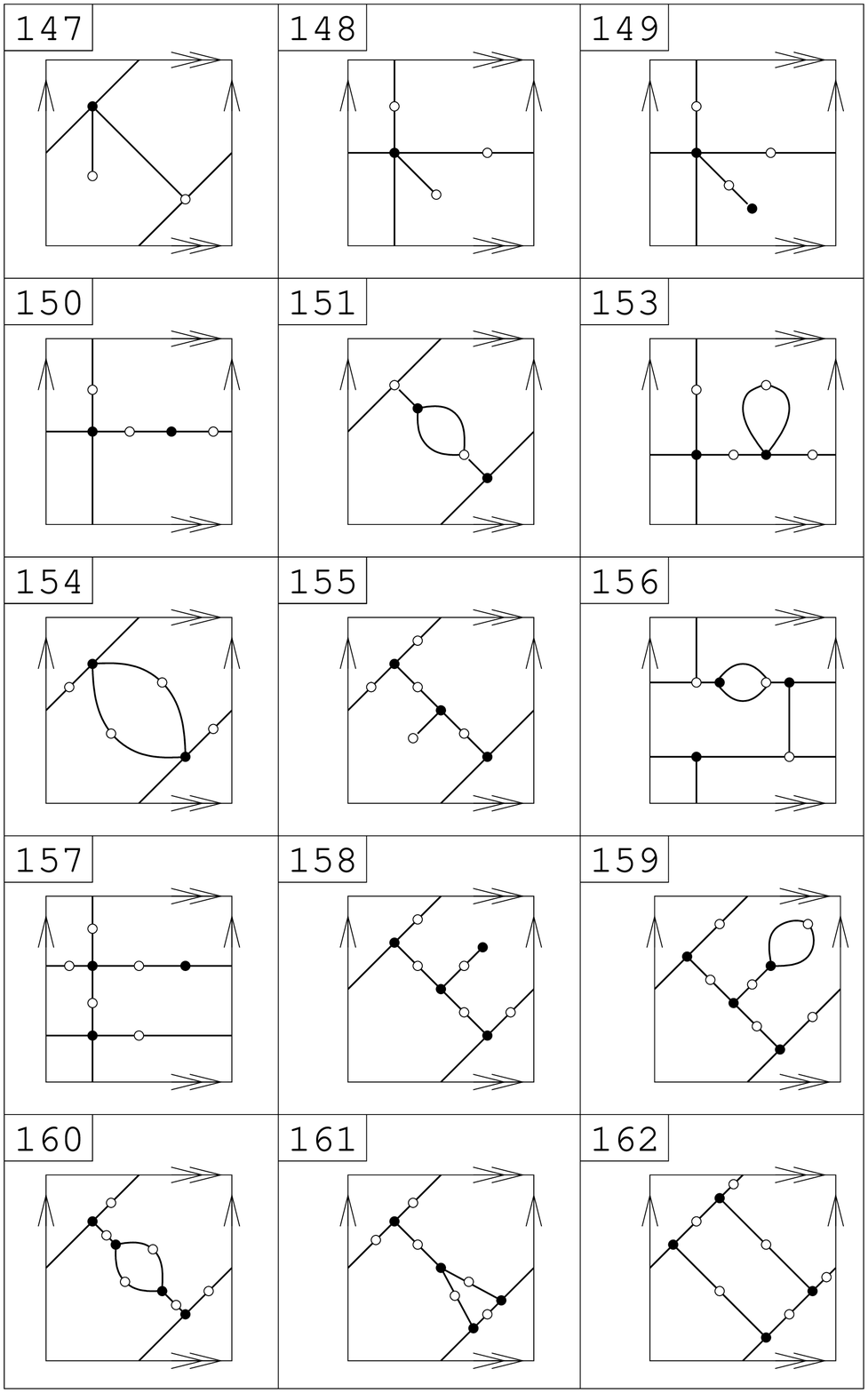}
\mycap{Dessins d'enfant realizing the
candidate covers with degree up to $12$.\label{torus:real:fig:1}}
\end{figure}
\begin{figure}
\centering
    \includegraphics[scale=.4]{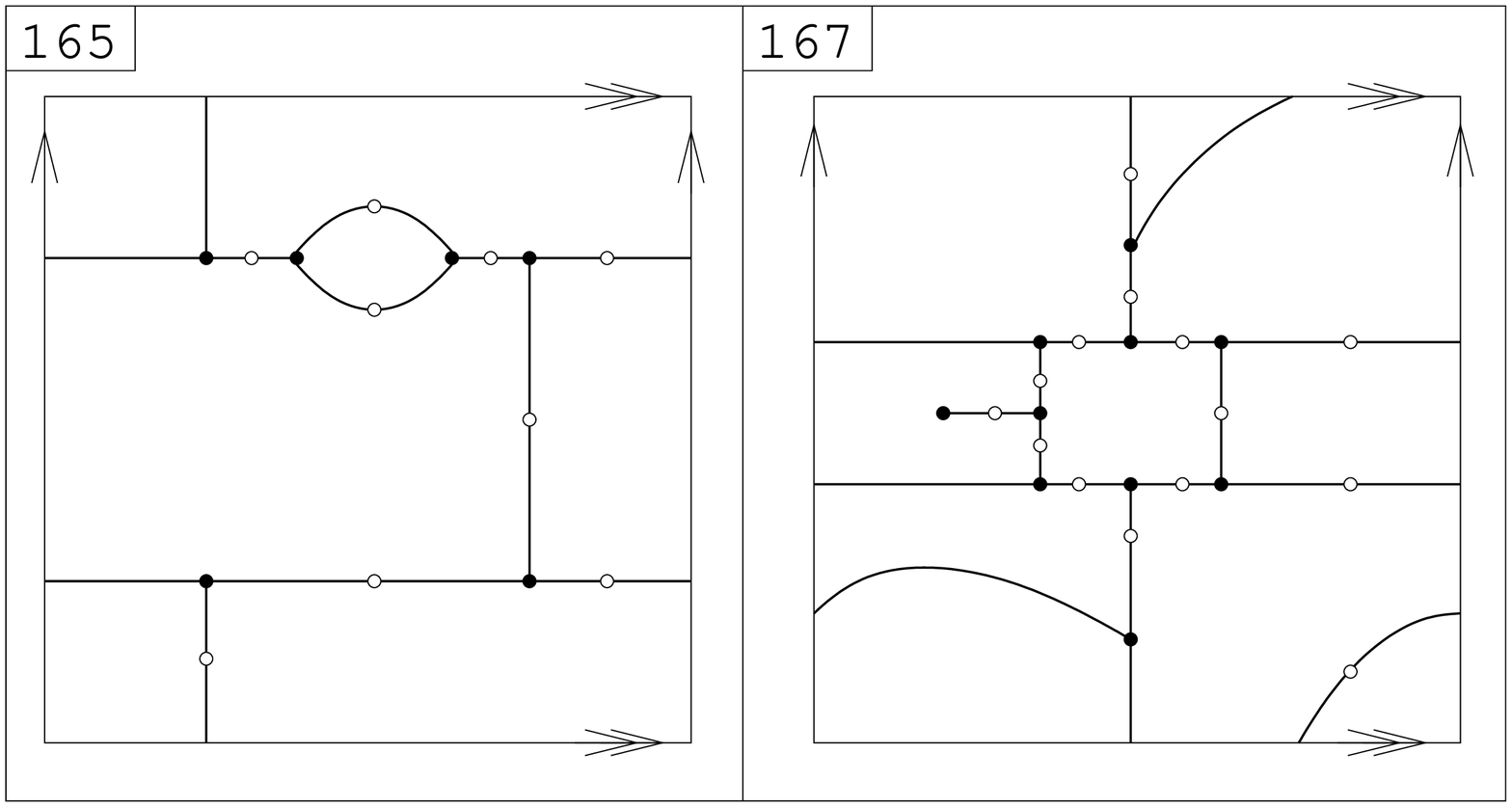}
\mycap{Dessins d'enfant realizing the
candidate covers with degree greater than $12$\label{torus:real:fig:2}}
\end{figure}
In all our dessins
we associate white vertices
to the entries of partition $\Pi_1$ and black vertices to those in $\Pi_2$,
so the regions of the complement correspond to the entries of $\Pi_3$.
\end{proof}

To conclude the proof of Theorem~\ref{hyp:T(1)-to-S(3):summary:thm}
one would now need to show exceptionality of candidates
\texttt{152}, \texttt{163}, \texttt{164}, \texttt{166} and \texttt{168}.
We have actually done this using the \textsf{MR} criterion and the code~\cite{Monge},
as explained in Section~\ref{techniques:sec}. We will however show here
the same fact in a geometric fashion for the smallest and for the largest candidates:

\begin{prop}\label{T(1):GG:prop}
Candidates \emph{\texttt{152}} and \emph{\texttt{168}} can be shown to be exceptional
using the \emph{\textsf{GG}} technique.
\end{prop}

\begin{proof}
For \texttt{152} we should realize $T(2)\dotsto^{6:1} S(3,3,4)$ with (of course) $2\to 4$.
Taking the cone orders $3,3,4$ at $A,B,C$ we should then assemble $12$ copies
of a triangle with vertices $A,B,C$ getting, after the gluing, vertices $A_1,A_2,B_1,B_2$
each adjacent to $6$ triangles. Taking into account this property
for $A_1,A_2,B_1$ we get the pattern first shown in Figure~\ref{152:fig}, and taking
into account also $B_2$ we get the two possible patterns also shown in the figure.
\begin{figure}
\centering
    \includegraphics[scale=.7]{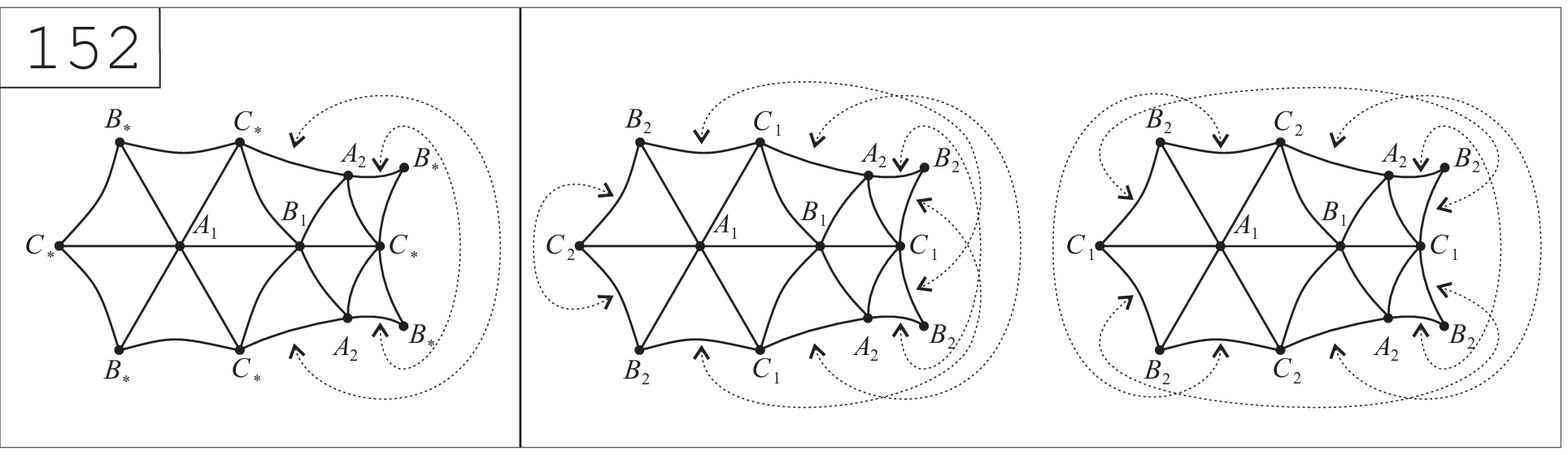}
\mycap{Exceptionality of \texttt{152} via \textsf{GG}.\label{152:fig}}
\end{figure}
These patterns do not give orbifolds, because
in the first one the cone angle at $C_1$ is $\frac{5\pi}2$, and
in the second one the cone angles at $C_1,C_2$ are $\frac{3\pi}2$.
Note however that combinatorially they give realization of candidate \texttt{151}
and of the Euclidean $T\dotsto^{6:1}S(3,3,3)$.

Turning to \texttt{168}, we confine ourselves to a sketch, because the complete argument is long.
We should realize $T(7)$ gluing $72$ copies of $T(2,3,7)$. Using the fact that $2$ is covered
by smooth points and $7$ is covered by $5$ smooth points and one of order $7$,
we see that the triangles must get assembled into $5$ heptagons and one monogon, all with angles
$2\alpha$, with $\alpha=\frac\pi3$. For short,
let us write $n$ instead of $n\alpha$.
We should now glue these blocks getting angle $6$
at all vertices. In particular when after a partial gluing an angle of $6$ is reached,
an extra gluing of edges is forced.
Gluing the monogon to one heptagon we see the latter gets replaced by square with
angles $2,2,2,4$. This square cannot glue to itself,
so another heptagon gets replaced by one with angles $4,2,4,2,2,2,2$
(in this order). Checking various possibilities one sees that an edge incident to an angle $4$ of this
heptagon cannot be glued to an edge of the same heptagon, so, from the 3 original heptagons left,
one gets replaced by a decagon with angles $4,4,2,2,2,4,2,2,2,2$.
The edge of the decagon with angle $4$ at both ends cannot be glued to another edge of the decagon,
so one of the two remaining original heptagons gets replaced by an $11$-gon with angles
$2,4,2,2,2,4,2,2,2,4,2$.
At least one edge of the heptagon gets glued to an edge of this $11$-gon, and there are $6$ possibilities up to symmetry.
Looking at them we reduce to a single block, that is either a $14$-gon with angles
\begin{center}$2,2,4,2,2,4,2,2,4,2,2,4,2,2\quad\text{or}\quad2,2,4,2,2,4,2,2,4,2,2,2,4,2$\end{center}
or a $16$-gon with angles
\begin{center}$4,2,2,2,2,2,4,4,2,2,2,4,2,2,4,2\quad\text{or}\quad4,2,2,2,2,2,4,4,2,2,4,2,2,2,4,2$.\end{center}
And with some patience one sees that from one such block it is impossible to get
a torus imposing angle $6$ at all vertices.
\end{proof}

\vspace{.35cm}

\noindent
Scientific Visualization Unit\\
Institute of Clinical Physiology - CNR\\
Via G. Moruzzi 1\\
56124 Pisa, Italy\\
mantonietta.pascali@gmail.com

\vspace{.35cm}

\noindent
Dipartimento di Matematica Applicata\\
Universit\`a di Pisa\\
Via Filippo Buonarroti, 1C\\
56127 PISA, Italy\\
petronio@dm.unipi.it

\end{document}